\documentclass[11pt]{article}

\usepackage{amssymb}

\usepackage{amsmath,amsthm,amssymb}
\usepackage{times}
\usepackage{enumerate}
\usepackage{amsmath}
\usepackage[all, cmtip]{xy}
\usepackage{amscd}
\usepackage{graphicx}
\usepackage{geometry}

\theoremstyle{plain}
\newtheorem{thm}{Theorem}[section]
\newtheorem{corollary}[thm]{Corollary}
\newtheorem{lemma}[thm]{Lemma}
\newtheorem{proposition}[thm]{Proposition}
\newtheorem*{thmintro}{Theorem \ref{stable_preserving}}
\newtheorem*{thmintro2}{Theorem \ref{Gamma_0}}
\newtheorem*{thmintro3}{Theorem \ref{compactification}}
\newtheorem{theorem}[thm]{Theorem}

\theoremstyle{definition}
\newtheorem{definition}[thm]{Definition}

\newtheorem{condition}{Condition}

\theoremstyle{remark}

\newtheorem{example}[thm]{Example}

\newtheorem{remark}{Remark}[section]
\newtheorem*{notation}{Notation}

\newcommand{\s}[1]{\ensuremath{\mathcal{#1}}}
\newcommand{\mN}{\ensuremath{\mathbb{N}}}
\newcommand{\mC}{\ensuremath{\mathbb{C}}}
\newcommand{\mZ}{\ensuremath{\mathbb{Z}}}
\newcommand{\mR}{\ensuremath{\mathbb{R}}}
\newcommand{\mP}{\ensuremath{\mathbb{P}^{1}}}
\newcommand{\phom}{\ensuremath{\mathrm{Hom}}}
\newcommand{\rhom}[1]{\ensuremath{\mathrm{RHom}^{#1}}}
\newcommand{\bderive}[1]{\ensuremath{\mathrm{D}^{\mathrm{b}}(#1)}}

\newcommand{\coh}[1]{\ensuremath{\mathcal{H}^{#1}}}
\newcommand{\mr}[1]{\ensuremath{\mathrm{#1}}}

\newcommand{\image}{\operatorname{im}}
\newcommand{\rktot}{\operatorname{rk}_{tot}}

\newcommand{\coimage}{\operatorname{coim}}

\newcommand{\imaginary}{\operatorname{\ensuremath{\mathfrak{I}m}}}
\newcommand{\pr}{\prime}
\newcommand{\sP}{\s{P}}
\newcommand{\sE}{\s{E}}
\newcommand{\sA}{\s{A}}

\newcommand{\sK}{\s{K}}
\newcommand{\sL}{\s{L}}
\newcommand{\sV}{\s{V}}
\newcommand{\sO}{\s{O}}
\newcommand{\sF}{\s{F}}
\newcommand{\sG}{\s{G}}

\newcommand{\sS}{\s{S}}
\newcommand{\sT}{\s{T}}

\newcommand{\Stab}{\operatorname{Stab}}
\newcommand{\Aut}{\operatorname{Aut}}
\newcommand{\End}{\operatorname{End}}

\newcommand{\Coh}{\operatorname{Coh}}
\newcommand{\stab}{\ensuremath{(Z, \s{P})}}
\newcommand{\E}{\ensuremath{\mathbb{E}}}
\newcommand{\I}{\ensuremath{\mathbb{I}}}
\newcommand{\rank}{\operatorname{rank}}

\newcommand{\Pic}{\operatorname{Pic}}
\newcommand{\wtpi}{\widetilde{\pi}}

\newcommand{\shortexactseq}[3]{\ensuremath{ 0 \rightarrow #1 \rightarrow #2 \rightarrow #3 \rightarrow  0}}

\newcommand{\bxi}{\bar{\xi}}

\title{Interactions between autoequivalences, stability conditions, and moduli problems.}

\author{Parker E. Lowrey}

\begin{document}

\maketitle

\begin{abstract}
We formulate a strong compatibility between autoequivalences and Bridgeland stability conditions and derive a sufficiency criterion.  We apply this criterion to an extension of classical slope on the derived category associated to any Galois cover of the nodal cubic.  In particular, we give an explicit description of the moduli space of stable objects, its compactification (under S-equivalence), and the group of all autoequivalences compatible with the choice of stability condition.


\end{abstract}

\section{Introduction}
\label{introduction-chapter}
Triangulated categories appear throughout geometry, with the different incarnations containing varying geometric data.  A well studied example is $\bderive{X}$,  the bounded derived category of coherent sheaves on a (locally Noetherian) scheme $X$.  This paper is centered on using interactions between stability conditions and exact autoequivalences to derive information about each other, revealing a great deal of geometric information about the underlying triangulated category in the process.  In particular, we study when an autoequivalence and a stability condition are well adapted to each other.  

The prominent example of how a well adapted autoequivalence can elucidate structure in a triangulated category is the Fourier-Mukai transform on an elliptic curve, $E$.    This is an autoequivalence obtained through the integral transform $\Phi_{\s{P}}$ of $\bderive{E}$ with the kernel $\s{P}$ isomorphic to the Poincar\'e bundle.  This autoequivalence does not preserve the geometric t-structure $\mr{Coh}(E)$,  yet it is known that $\Phi_{\s{P}}$ is highly compatible with all stability conditions on $\bderive{E}$, including classical slope stability.  One can describe its action on $\bderive{E}$ through the equality $\Phi_{\s{P}} \cdot (Z, \s{P}) = (Z, \s{P}) \cdot e^{i\pi/2}$ on the stability manifold.  This description allows one to reinterpret the method used in Atiyah's classification of the moduli of semistable vector bundles on $E$ as a transitive action of $\mr{SL}(2, \mZ)$ (via $\Phi_{\s{P}}$ and $\otimes \sL$ for any principle polarization $\s{L}$) on the set of phases of $(Z, \sP)$ \cite{MR0131423, MR2264108}.


With this example in mind, we define an autoequivalence $\Phi$ and a stability condition $\stab$ to be compatible if 
$\Phi$ preserves semistable objects and Harder-Narasimhan filtrations.  We developed the following criterion giving sufficiency conditions for compatibility.
\begin{thmintro}
Let $\Phi \in \mr{Aut}(\s{T})$, where $\s{T}$ is a triangulated category.  Given $(Z, \s{P}) \in \mr{Stab}(\s{T})$, a locally finite stability condition, one has the  t-structure $\s{D}^{\leq 0} := \s{P}(0, \infty)$, with heart $\s{A}$.
If $\Phi$ satisfies
\begin{enumerate}
\item[(i)] $\Phi(\s{A}) \subset \s{D}^{\leq M} \cap \s{D}^{\geq M -1}$, $M \in \mZ$.
\item[(ii)] $\Phi$ descends to an automorphism $[[\Phi]]$ of $\image(Z) \cong K(\s{T})/\ker(Z)$.
\item[(iii)] $\Phi$ preserves the ordering on $\image(Z)_{eff,comp} \subset \image(Z)$ induced by $(Z, \s{P})$.
\end{enumerate}
then $\Phi$ is compatible with $\stab$.
\end{thmintro}
\noindent Here, $\image(Z)_{eff, comp}$  (defined in \S\ref{compatible_chapter}) is a minimal subset for which we must check that $\Phi$ strictly preserves the ordering.  The first condition is a necessary one, a fact that will be evident from the definitions. The beauty of this criterion is its ability to reduce compatibility to a series of linear algebra calculations.

To demonstrate the usefulness of our criterion we apply it to $\bderive{\E_n}$, where $\E_n$  ($n \in \mathbb{N}^{>0}$) denotes a singular scheme defined by simple combinatorial data:  $\E_n$ is a singular reducible genus 1 curve that can be envisioned as a cycle of $n$ projective planes with transverse intersections.  More explicitly, they are Galois covers of the Weierstrass nodal cubic.    The singular and non-irreducible nature of these curves limits the application of standard technology used to work with a derived category.

It is well known that given an elliptic curve E,  $\Aut(\bderive{E})$ is an extension of  $\mr{SL}(2, \mZ)$.  This fact is central to many results pertaining to elliptic curves.  Using our criterion, we provide an analog for $n$-gons.    Let $Z_{cl}: K(\E_n) \rightarrow \mC$ be defined by $Z_{cl}(F) = -\chi(F) + i(\Sigma_{0 < j \leq n} \mr{rk}_{j}(F))$. The function $Z_{cl}$ and heart $\mr{Coh}(\E_n)$ form a stability condition $\sigma_{cl}(n) := (Z_{cl}, \s{P}_{cl})$.  This stability condition is the obvious extension of classical slope to $n$-gons.
\begin{thmintro2}
	Given $n$, let $\Aut_{cl}(n) \subset \Aut(\bderive{\E_n})$ consist of all autoequivalences compatible with $(Z_{cl}, \s{P}_{cl})$.  Then 
	\begin{displaymath}
		\xymatrix{ 1 \ar[r] & ((\mC^\ast)^n \rtimes D_{2n}) \times \mZ \times \mC^\ast  \ar[r] &  \Aut_{cl}(n) \ar[r] & \Gamma_0(n) \ar[r] & 1}
	\end{displaymath}
 where $\Gamma_0(n) \subset \mr{SL}(2, \mZ)$.  The leftmost group is generated by $\Aut(\E_n)$, $\mr{Pic}^0(\E_n)$, and the double shift $[2]$. Under the action of $\Aut_{cl}(n)$ there are $\Sigma_{d|n, d > 0} \phi(\gcd(d, \frac{n}{d}))$ equivalence classes of phases, where $\phi$ is Euler's function.
\end{thmintro2}
\noindent This is a departure from the niceties of elliptic curves and nodal cubics: there is one equivalence class in those cases.  Further, the modular groups $\Gamma_0(n)$ are not unipotently generated for most $n$ \footnote{Private correspondence with Daniel Allcock}.  Since compatible spherical twists act unipotently (on $\image(Z_{cl})$), this means $\Aut_{cl}(n)$ is not generated by compatible spherical twists for large $n$.

To demonstrate how compatible autoequivalences can reveal information about stability conditions, we provide an extension of Atiyah's classification to the $n$-gon.  The semistable objects in $(Z_{cl}, \s{P}_{cl})$ correspond to Simpson semistable objects for a polarization on $\E_n$ \cite{Ruiperez_Lopez, Simpson_moduli}.  Let $\s{M}_{cl}(n,a)$ denote the the coarse moduli space of semistable objects with phase $a$.  The analog of Atiyah's work is
\begin{thmintro3}
Given a nontrivial slice $\s{P}(a)$ of $(Z_{cl}, \s{P}_{cl})$ on $\E_n$ ($n > 1$), let $\s{M}^{st}_{cl}(n,a) \subset \s{M}_{cl}(n,a)$ denote the subscheme whose closed points correspond to stable objects and $\overline{\s{M}^{st}_{cl}(n,a)}$ its closure.  Then $\overline{\s{M}^{st}_{cl}(n,a)} \cong \E_s \coprod \mZ/n\mZ$ where $s|n$ ($s$ depends on $a$) and each component of $\overline{\s{M}^{st}_{cl}(n,a)}$ is a component of $\s{M}_{cl}(n,a)$.
\end{thmintro3}
\noindent Here we restrict to the case of $n > 1$ since for the nodal cubic it is known that $\s{M}^{st}_{cl}(1,a) \cong \E_1$ \cite{bk_stab}.  The case of $n = 2$ was calculated in \cite{Ruiperez_Lopez}.

The disconnectedness of $\overline{\s{M}_{cl}^{st}(n,a)}$ is a result of the rigidness of indecomposable torsion-free, but not locally free, sheaves on $\E_n$.  If $a$ is such that locally free objects $\s{V} \in \s{M}_{cl}^{st}(n,a)$ are line bundles these ``rigid" elements correspond to line bundles restricted to $s$ consecutive components.  There is an action by the Galois group of $\pi_n: \E_n \rightarrow \E_1$ on $\overline{\s{M}_{cl}^{st}(n,a)}$ that cyclically permutes the rigid elements and factors through $\mr{Gal}(\E_s \rightarrow \E_1)$ on the positive dimensional component.  Our classification, like Atiyah's,  relies on the reduction of the problem to a finite number of specific cases.  This is provided by Theorem \ref{Gamma_0}.  The shortcoming of Atiyah's method applied to the $n$-gon is now clear: for $n > 4$ taking quotients of semistable sheaves by trivial subbundles is not enough to reduce the classification problem to a finite number of cases.  One must ``reduce"  by an infinite number of sheaves to get a finite number of classes.  This highlights the importance of having a general criterion.  

Although we restrict to stable objects in Theorem \ref{compactification}, the methods used can be extended to classify the entire coarse moduli space of semistable objects.  This is done by comparing $\s{M}_{cl}(n, a)$ to $\s{M}_{cl}(s, 1)$ for a suitable $s$.  The structure of the latter space is completely known: $\s{M}_{cl}(s, 1) \cong \coprod \mr{Sym}^r(\E_s)$.  Using this one should be able to write $\s{M}_{cl}(n, a)$ in terms of $\coprod \mr{Sym}^r(\E_s)$.   This work should allow substantial progress in classifying moduli spaces obtained from other Simpson stability conditions and more general genus 1 curves.

We conclude this introduction with a quick discussion on the definition of a compatible autoequivalence, and how this definition influences Theorem \ref{stable_preserving}.  As a starting point, the most general compatibility between an autoequivalence and a stability condition is simply that our autoequivalence preserves semistable objects.  Without more knowledge about the categorical structure of our triangulated category, finding easily verifiable conditions is not possible.  Ideally, we want to avoid almost all knowledge of our triangulated category in the calculations.  To obtain such low level conditions, one must consider more than just semistable sheaves.  A stability condition gives a binary relation on semistable sheaves, induced by the charge on $K(\s{T})$.   To take this data into account, we define a compatible autoequivalence as a semistable preserving autoequivalence that preserves this binary relation.  This new definition has the following strong consequence: if $a := \phi(\Phi(F))$ for any semistable object $F \in \s{P}(1)$, then $\Phi$ induces an equivalence $\s{P}(0, 1] \cong \s{P}(a-1, a]$.  To understand the origin of the criterion, observe that the equivalence gives $\Phi(\s{P}(0, 1]) \subset \s{P}(a, a+1] \subset \Phi(M, M-2)$ for some $M \in \mZ$.  Thus $\Phi(\s{P}(0, 1])$ is concentrated in a maximum of 2 cohomological degrees (with respect to the original t-structure).  This shows that Theorem \ref{stable_preserving}(i) is a necessary condition.  Moreover, with (i) and (ii) in place, it is clear that a compatible autoequivalence will satisfy (iii).   From this perspective, (ii) is the only strong assumption in our criterion.  It is this assumption that allows one to use the full data supplied by the stability condition.  This in turn allows us to ignore most categorical structure in our criterion.

The layout of this paper is as follows.  \S\ref{background_chapter} is an overview of stability conditions,  the inherent group action on their moduli, and integral transforms.  With these preliminaries out of the way, in \S\ref{compatible_chapter} we begin our discussion of compatibility between an autoequivalence and a stability condition.  We show some easy consequences of the definitions and begin formulating properties that we can expect compatible autoequivalences to have.  These basic properties give the criterion in Theorem \ref{stable_preserving}, and the section is concluded with the proof of Theorem \ref{stable_preserving}.  \S\ref{n-gons_chapter} contains a review of basic terminology and results about $\bderive{\E_n}$.  With the terminology set, in \S\ref{gamma_chapter} we use Theorem \ref{stable_preserving} to construct compatible autoequivalences of $\bderive{\E_n}$, study basic properties of the group they generate.  Lastly, in \S\ref{applications_chapter} we use our compatible autoequivalences to understand aspects of the stability condition $(Z_{cl}, \s{P}_{cl})$: the compactification of the moduli of stable objects of a given phase, and number of phases for a given $n$.

\subsection{A brief note on notation}
\label{notation}
Throughout this paper $\s{T}$ will be an essentially small triangulated category.  By $\mr{Aut}(\s{T})$ we mean the group of exact autoequivalences of $\s{T}$, i.e., autoequivalences that preserve the triangulated structure.

Given a set $A \subset \mC$ and a log branch $\log_\tau$ discontinuous at the ray $\mR^{\geq 0} e^{i\tau}$ with $\tau \in [0, 2\pi)$, we can assign a ``phase'' to each non-zero element of $A$: given $a \in A$ define $\phi_{f, \tau}(a) = \frac{\imaginary \log a}{\pi}$.  Clearly this depends on the log branch.  We will sometimes abbreviate $\phi_{f, 0}$ by $\phi_f$.    

We let $\mathbb{H}$ denote $\{z \in \mC| \imaginary z > 0\}$ (the standard upper half plane) and $\mathbb{H}^\prime := \mathbb{H} \cup \mR^{< 0}$.  Lastly, we let $\mu_n$ denote the group of $nth$ roots of unity.  When we we are working with an explicit generator, we will instead use $\mZ/nZ$.

\section{Background.}
\label{background_chapter}

\subsection{Stability conditions}
Stability conditions on triangulated categories were first defined and studied by Bridgeland in \cite{bridgeland_stab_triangulated}, \cite{bridgeland_spaces_stability}.  Although stability conditions are the culmination of a long standing attempt to extend geometric invariant theory (GIT) to the setting of triangulated categories, the original motivation for Bridgeland was to provide a mathematical basis for Douglas's $\Pi$-stability \cite{MR1909947, MR1940182}.

\subsubsection{Definition}
Given a triangulated category $\s{T}$, we define the K-group $K(\s{T})$ as the abelian group freely generated on $\mr{Iso}(\s{T})$ (isomorphism classes of objects of $\sT$),  subject to the relations $[B] = [A] + [C]$ if there exists a triangle $A \rightarrow B \rightarrow C \rightarrow A[1] \rightarrow \ldots$.

\begin{definition}\cite{bridgeland_stab_triangulated}
\label{stab defn}
A stability condition $(Z,\s{P})$ on a triangulated category $\s{T}$ consists of a group homomorphism
$Z: K(\s{T}) \rightarrow \mathbb{C}$ called the central charge, and full additive subcategories $\s{P}(\phi)\subset\s{T}$ for each $\phi\in\mathbb{R}$, satisfying the following axioms:
\begin{itemize}
\item[(a)] if $E\in \s{P}(\phi)$ then $Z(E)= m(E)\exp(i\pi\phi)$ for some
 $m(E)\in\mathbb{R}_{>0}$,
\item[(b)] for all $\phi\in\mathbb{R}$, $\s{P}(\phi+1)=\s{P}(\phi)[1]$,
\item[(c)] if $\phi_1>\phi_2$ and $A_j\in\s{P}(\phi_j)$ then $Hom_{\s{T}}(A_1,A_2)=0$,
\item[(d)] for each nonzero object $E\in\s{T}$ there is a finite sequence of real
numbers
\[\phi_1>\phi_2> \cdots >\phi_n\]
and a collection of triangles
\[
\xymatrix@C=.4em{
0_{\ } \ar@{=}[r] & E_0 \ar[rrrr] &&&& E_1 \ar[rrrr] \ar[dll] &&&& E_2
\ar[rr] \ar[dll] && \ldots \ar[rr] && E_{n-1}
\ar[rrrr] &&&& E_n \ar[dll] \ar@{=}[r] &  E_{\ } \\
&&& A_1 \ar@{-->}[ull] &&&& A_2 \ar@{-->}[ull] &&&&&&&& A_n \ar@{-->}[ull]
}
\]
with $A_j\in\s{P}(\phi_j)$ for all $j$.

\end{itemize}

\end{definition}

The triangles in Definition \ref{stab defn}(d) are called the Harder-Narasimhan filtration of the object $E$.   
We designate the semistable objects of this filtration by $A_{j}(E)$, or $A_{j}$ when no confusion can arise.  If the number of semistable factors is irrelevant to our context we denote the lowest phase as $\phi_-$ and the accompanying semistable factor as $A_-$.

We denote by $\s{P}(a,b]$ the full extension-closed subcategory of $\s{T}$ generated by $\s{P}(c) \subset \s{T}$, $c \in (a, b]$. It is clear from Definition \ref{stab defn}(b), that all semistable objects, up to shift, are determined by any interval of length 1.  In fact,  the subcategory $\s{P}(a,a+ 1]$  is the heart of the  bounded t-structure given by the subcategory $\s{P}(a, \infty)$ (and therefore abelian).  The subcategory $\s{P}(0,1]$ will play a special role.  Our charge, $Z$, has the property that $Z([A]) \in \mathbb{H}^\prime$ for $A \in \sP(0,1]$. Since $\mathbb{H}^\prime$ is contained in a proper open domain of $\mC^{\ast}$, we can assume the existence of a log branch on $\mathbb{H}^\prime$ such that $\frac{1}{\pi} \imaginary log$ agrees with the phases in $(0, 1]$.  We can use this log branch to assign phases to all objects in $\s{P}(0,1]$, including those that are unstable.  This fact will be used extensively in our proof of Theorem \ref{stable_preserving}.

We can alternatively define a stability condition with the data $(\sA, Z)$ where $\sA$ is the heart of a bounded t-structure and a homomorphism (called a stability function) $Z: K(\s{A}) \rightarrow \mathbb{H}^\prime$ which satisfies a certain finiteness condition similar to Definition \ref{stab defn}(d).  To recover the above definition, set $\sP(0, 1] = A$ and calculate semistable objects via a slope function defined by $Z$.


\begin{example}
Let $X$ be an smooth algebraic curve.  Define $Z(\sF) = -\mr{deg}(\sF) + i\rank(\sF)$.   Then $\Coh(X)$ and $Z$ define a stability condition on $\bderive{X}$.   The semistable objects are those calculated by slope.
\end{example}

\subsubsection{Group actions on the moduli of stability conditions}
Let $\Phi \in \mr{Aut}(\s{T})$.  The relations on $K(\s{T})$ depend only on the triangulated structure of $\s{T}$, thus, $\Phi$ descends to an automorphism $[\Phi]$ of $K(\s{T})$.   The following lemma explicitly details a commuting left/right group action on the moduli of stability conditions.  We include the proof since this paper is partially concerned with relationships between these two actions.
\begin{lemma}[\cite{bridgeland_stab_triangulated} Lemma 8.2]
\label{groupactions}
The moduli space of (locally finite) stability conditions $\Stab(\s{T})$ carries a right action of the
group
$\widetilde{\mr{GL}}^+(2, \mR)$, the universal
covering space of $\mr{GL}(2,\mR)$, and a left action by the
group
$\Aut(\s{T})$ of exact autoequivalences of $\s{T}$. These two actions
commute.
\end{lemma}

\begin{proof}
First note that the group $\widetilde{\mr{GL}}^+(2, \mR)$ can be thought of as the
set of pairs $(T,f)$ where $f\colon \mR\to\mR$ is an
increasing map with $f(\phi+1)=f(\phi)+1$, and  $T\colon \mR^2\to\mR^2$
is an orientation-preserving linear isomorphism, such that the induced maps on
$S^1=\mR/2\mZ=\mR^2/\mR_{>0}$ are the same.

Given a stability condition $\sigma=(Z,\s{P})\in\Stab(\s{T})$, and a
pair $(T,f)\in\widetilde{\mr{GL}}^+(2, \mR)$, define a new stability
condition $\sigma'=(Z',\s{P}')$ by setting $Z'=T^{-1}\circ Z$ and
$\s{P}'(\phi)=\s{P}(f(\phi))$. Note that the semistable objects of the
stability conditions $\sigma$
and $\sigma'$ are the same, but the phases have been
relabeled.

For the second action, we know that an element
$\Phi\in\Aut(\s{T})$ induces an automorphism $[\Phi]$ of $K(\s{T})$. If
$\sigma=(Z,\s{P})$ is a stability condition on $\s{T}$ define $\Phi(\sigma)$
to be the stability condition $(Z\circ[\Phi]^{-1},\s{P}')$, where
$\s{P}'(t)$=$\Phi(\s{P}(t))$. The reader can easily check that this action
is by isometries and commutes with the first.
\end{proof}

We think of the action of $\widetilde{\mr{GL}}^+(2, \mR)$ as a ``change of coordinates'':  it is a change of basis on the range of our charge $Z$, followed by an appropriate reassigning of the phases  so that Definition \ref{stab defn} is satisfied.  An alternative way to understand the left group action is via the abelian subcategory/stability function viewpoint.  In particular, our autoequivalence $\Phi \in \Aut(\s{T})$ takes $\s{A}$ to another abelian category, $\Phi(\s{A})$.  One then obtains a stability function on $\Phi(\s{A})$ by pullback, $\Phi\cdot Z = Z([\Phi^{-1}]( \bullet))$.

\subsection{Integral transforms}
Let $X$ be a locally Noetherian scheme.  For any object $\sK \in \bderive{X \times X}$, using the natural projections
\begin{displaymath}
\xymatrix@C=.7em@R=.6em{ & X \times X \ar[ld]_{\rho_1} \ar[rd]^{\rho_2} & \\
X & & X }
\end{displaymath}
one can define an exact functor $\Phi_{\sK} \in \End(D^-(X))$ called an integral transform with kernel 
$\sK$.  Explicitly, 
\begin{displaymath}
\Phi_{\sK}(F) = \rho_{2\ast}(\rho_1^{\ast}(F)\otimes^{L} \sK).
\end{displaymath}
In the case that $\sK$ is either perfect or flat over $X$ (under the map $\rho_1$), we can restrict this to an endomorphism of $\bderive{X}$, see \cite{RMdS09} or \cite{bk_fm_fib}.  Given two integral transforms, we have the composition formula $\Phi_{\sK} \circ \Phi_{\sK^{\prime}} \cong \Phi_{\sK^{\prime} * \sK}$ where $\sK^{\prime} * \sK := \rho_{13\ast}(\rho_{12}^{\ast}(\sK^{\prime})\otimes^{L}\rho_{23}^{\ast}(\sK))$, with $\rho_{ij}$ the obvious projections of $X \times X \times X$  \cite{MR607081}.

The next proposition gives commutation relations between integral transforms and automorphisms of $X$.  This result is well known, and proven by standard techniques.
\begin{proposition}
\label{geometric_automorphism_relation}
Let $\Phi_{\s{K}}$ denote the endomorphism of $\bderive{X}$ with kernel $\sK \in \bderive{X \times X}$.  If $\alpha$ is an automorphism of $X$, then
\begin{enumerate}
\item $\alpha_{\ast} \circ \Phi_{\s{K}} \cong \Phi_{(\mr{Id} \times \alpha)_{\ast}\s{K}}$
\item $\Phi_{\s{K}} \circ \alpha_{\ast} \cong \Phi_{(\alpha \times \mr{Id})^{\ast}\s{K}}$
\item $\Phi_{\s{K}} \circ \alpha^{\ast} \cong \Phi_{(\alpha \times \mr{Id})_{\ast}\sK}$
\item $\alpha^{\ast} \circ \Phi_{\s{K}} \cong \Phi_{(\mr{Id} \times \alpha)^{\ast}\s{K}}$
\end{enumerate}
\end{proposition}



\section{Compatible Autoequivalences}

Throughout this section $\sT$ will be a triangulated category with an exact autoequivalence $\Phi$ and stability condition $\stab$.  We denote the heart of the t-structure $\s{D}^{\leq 0} = \s{P}(0, \infty)$ by $\s{A} = \s{P}(0, 1]$. 

\subsection{Definition of compatibility}
\label{compatible_chapter}
\begin{definition}
\label{compat_defn}
An autoequivalence $\Phi$ of $\sT$ is compatible with $\stab$ if
\begin{enumerate}
\item $\Phi(F)$ is semistable for every semistable $F$,
\item $\Phi$ preserves all Harder-Narasimhan filtrations, i.e., if $\{A_i\}_{1 \leq i \leq n}$ belongs to a Harder-Narasimhan filtration for $E$, then $\{\Phi(A_i)\}_{1 \leq i \leq n}$ belongs to a Harder-Narasimhan filtration for $\Phi(E)$.
\end{enumerate}
\end{definition}

The coherence between $\Phi$ and $\stab$ supplied by the second condition is the focus of this paper.  It allows for a nice comparison between $\stab$ and $\Phi(\stab)$. For example, it ensures no unstable objects are mapped to stable objects; something the first condition cannot guarantee.  An autoequivalence that satisfies Definition \ref{compat_defn}(i) is known in literature as ``semistable preserving''.



The definition of compatibility has an alternate characterization:
\begin{proposition}
 \label{compat_alt_defn}
  $\Phi$ is compatible with $\stab$ if and only if it is semistable preserving and
  \begin{enumerate}
  \item $\phi(F) < \phi(G)$ implies $\phi(\Phi(F)) < \phi(\Phi(G))$,
  \item $\phi(F) = \phi(G)$ implies $\phi(\Phi(F)) = \phi(\Phi(G))$.
  \end{enumerate}
\end{proposition}
\begin{proof}
 If $\Phi$ satisfies the conditions of the proposition then it is clearly compatible with $\stab$.  For the other direction, if $F, G$ are semistable with $\phi(F) < \phi(G)$, $F\oplus G$ has a canonical Harder-Narasimhan filtration with $A_1 = F, A_2 = G$.  Since $\Phi$ is compatible, $\phi(\Phi(F)) < \phi(\Phi(G))$.  If $\phi(F) = \phi(G)$ then $F\oplus G$ is semistable and thus $\Phi(F\oplus G)$ as well.  Since $\Phi(G)$ is a summand of $\Phi(F\oplus G)$, this forces $\phi(\Phi(F)) = \phi(\Phi(G))$.      
\end{proof}
We abbreviate these two conditions as follows: $\phi(F) $ {\tiny $\stackrel{<}{(=)}$} $\phi(G)$ implies $\phi(\Phi(F)) $ {\tiny $\stackrel{<}{(=)}$} $\phi(\Phi(G))$.  The ordering of slices imparts a binary relation on semistable objects (it is not a linear order since it doesn't satisfy antisymmetry).  By abuse of terminology, we will refer to this binary relation as an ordering.  We now list several immediate consequences of Definition \ref{compat_defn} and Proposition \ref{compat_alt_defn}.
\begin{proposition}\mbox{}
\label{properties}
The composition of compatible autoequivalences are compatible.  Further, if $\Phi$ is compatible with \stab, then
\begin{enumerate}
\item for any nontrivial slice $\sP(a)$, $\Phi$ induces an equivalence $\Phi: \s{P}(a) \xrightarrow{\cong} \s{P}(b)$ for some $b \in \mR$,
\item $\Phi(F)$ is semistable if and only if $F$ is semistable,
\item  $\Phi(F)$ is stable if and only if $F$ is stable,
\item $\Phi$ is compatible if and only if $\Phi^{-1}$ is compatible.
\end{enumerate}
\end{proposition}
\begin{proof}
The composition statement is obvious.  For (1) the requirements for compatibility ensure $\Phi$ preserves semistable objects, and strictly preserves the ordering.   We can recover the semistable objects of a particular slice via this relation (short of the zero object).  Since slices are full subcategories, the result follows from $\Phi$ being an equivalence.  
(2) is immediate from the proof of Proposition \ref{compat_alt_defn}.  (3) follows from (1) and (2) since stable objects of phase $\alpha$ are the simple objects in the abelian category $\s{P}(\alpha)$.  

Lastly, for (4) note that $\Phi^{-1}$ is semistable preserving by (2).  The second condition of Proposition \ref{compat_alt_defn} is immediate from (2) and the compatibility of $\Phi$.
\end{proof}

There is a stronger notion of compatibility that will play an important role in the application of our criterion.  This stronger notion incorporates the structure of the left and right group actions on $\Stab(\sT)$. 
\begin{definition}
$\Phi$ is strongly compatible if $\Phi\cdot(Z, \s{P}) = (Z, \s{P})\cdot g$ for some $g \in \widetilde{\mr{GL}}^{+}(2, \mR)$.
\end{definition}
\subsection{A criterion for compatibility}
Our goal in this section is to obtain a criterion to test compatibility with conditions verifiable from $K(\s{T}$ and the natural t-structure given by $stab$.  To motivate the criterion, we start by assuming $\Phi$ is compatible with $\stab$ and derive various facts about $\Phi$. 

Proposition \ref{properties} implies the existence of $b \in \mR$ such that $\Phi(\s{P}(0)) \subset \s{P}(b)$.  Chasing through definitions reveals $\Phi(\s{P}(1)) \subset \s{P}(b+1)$.  The order preserving property of $\Phi$ implies $\Phi$ restricts to an equivalence: $\Phi(\sP(0, 1])) \xrightarrow{\cong} \sP(b, b+1]$.  Thus, $\Phi(\s{A})$ can be concentrated in a maximum of 2 consecutive cohomological degrees.  Our first condition is
\begin{condition}
\label{cohomological_condition}
 $\coh{i}(\Phi(\s{A})) = 0$ if $i \neq \{M, M-1\}$ where $M \in \mZ$.
\end{condition}

Without loss of generality, we assume $M = 0$ (we can compose with the compatible autoequivalence $[1]$).  With Condition \ref{cohomological_condition} in place, $\stab$ allows the choice of a log branch on $\mC^{\ast}$ discontinuous at $\mR^{> 0} e^{i0}$ having the property that the phases assigned by this branch agree with those in $\sP(0, 2]$.  We want to assign a phase $\phi_f(\Phi(F))$\footnote{The subscript $f$ is used to emphasize that the phase is defined regardless if $\Phi(F)$ is semistable, thus it is a ``fake'' phase} for all $F \in \s{A}$, not necessarily semistable.  Condition \ref{cohomological_condition} alone does not enable this: it is possible that $Z(\Phi(F)) = 0, F \in \sT$ with $F \neq 0$.   If we assume $[\Phi]$ descends to an automorphism $[[\Phi]]$ of $\image(Z)$, this issue is alleviated.  This motivates our second condition.

\begin{notation}
Given $[F] \in K(\sT)$, let $[[F]]$ denote its image in $\image(Z)$.  Likewise, if $A \in \End(K(\sT))$, let $[[A]]$ denote the resulting endomorphism on $\image(Z)$, if well defined. 
\end{notation}

\begin{condition}
$[\Phi]$ descends to an automorphism  $[[\Phi]]$ of $\image(Z)$.
\end{condition}
\begin{remark} It seems plausible that this condition may be stricter than needed to prove compatibility.  However, with a less stringent Condition (2) it is not likely that Condition (3) can be phrased in terms of K-theoretic calculations. \end{remark}

We can now assign our ``fake" phases:

\begin{definition}
\label{faux_phase}
With the choice of log branch above, given a nonzero $v \in \image(Z)$, define $\phi_{f}(v) := \frac{1}{\pi}\imaginary log(Z(v))$.  If $[[\sF]] = v$ for $\sF \in {\sT}$, we define $\phi_f(\sF)$ to be $\phi_f(v)$.
\end{definition}

Our log branch gives two orderings on the semistable objects of $\s{A}$: prior and subsequent to the application of $\Phi$. If $\Phi$ is compatible, these will be identical.  Since these orderings solely rely on $\image(Z) := K(\s{T})/\ker(Z)$ and the log branch, we will rephrase this in terms of a subset of $\image(Z)$.  First, in order to limit calculations, we restrict to the ``effective" subset:
\begin{displaymath}
\image(Z)_{eff} =  \{\pm v \in \image(Z)\;| \;\exists \textrm{ s.s. } G \in \s{A} \textrm{ with } [[G]] = v\}.
\end{displaymath}
The discontinuity of the log branch disallows meaningful comparison between the before and after orderings on $\image(Z)_{eff}$.  We further restrict to the largest subset of  $\image(Z)_{eff}$ for which this comparison {\it is} valid. To do so, we need to ensure that $[[\Phi]]$ doesn't move objects ``past the discontinuity''.   Define
\begin{align*}
\image(Z)_{comp} =  \{ v | \; Z(v) \in \mathbb{H}^\prime \}  \cup \{\pm v |\;  Z(v) \in \mathbb{H}^\prime \textrm{ and } Z([[\Phi]]v) \in \mathbb{H}^\prime\}
\end{align*}

\begin{example}
Suppose that $[[\Phi]]$ is the restriction of an element $\Upsilon \in \mr{GL}^+(2, \mR)$ (via $Z$).  Let $\s{G} = \mathbb{H}^{\prime} \cap \Upsilon^{-1}(\mathbb{H}^{\prime})$.  Then $\image(Z)_{comp}$ is the subset of $\image(Z)$ contained in $\mathbb{H}^\prime \cup -\s{G}$.

\begin{figure}[htb!]
\centering
\resizebox{!}{50mm}{\includegraphics{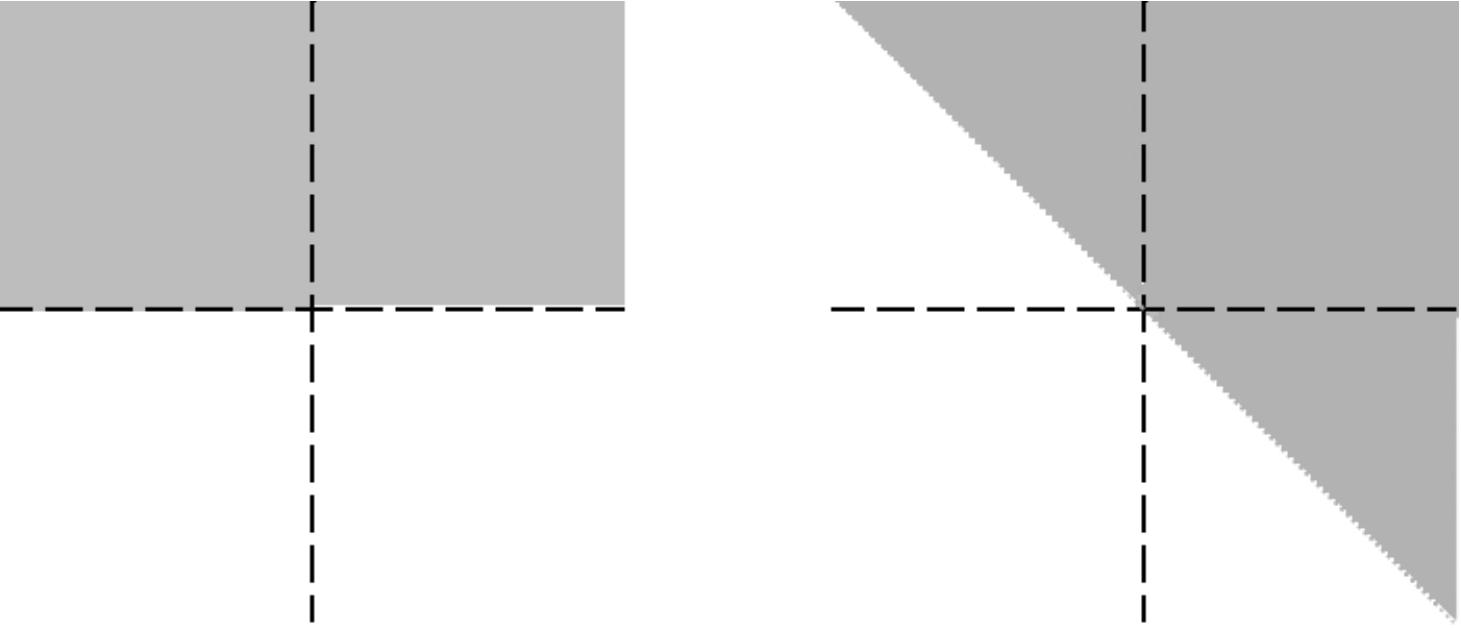}}
\caption{$\mathbb{H}^\prime$ and $\Upsilon^{-1}(\mathbb{H}^\prime)$}
\end{figure}
\begin{figure}[htb!]
\centering
\resizebox{!}{50mm}{\includegraphics{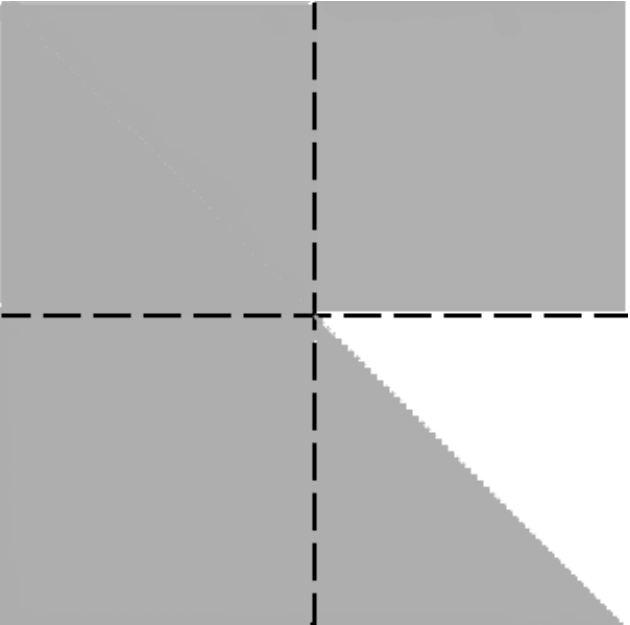}}
\caption{$\mathbb{H}^\prime \cup -\s{G}$}
\end{figure}
\end{example}

Since we can compare the two orderings on $\image(Z)_{comp}$, our final condition is that these orderings are the same, i.e., the ordering is strictly preserved under $[[\Phi]]$.
\begin{condition}
\label{order_preserving_condition}
Given $v,w \in \image(Z)_{eff \cap comp} := \image(Z)_{eff} \cap \image(Z)_{comp}$ if $\phi_f(v) $ {\tiny $\stackrel{<}{(=)}$} $\phi_f(w)$ then $\phi_f([[\Phi]](v))$ {\tiny $\stackrel{<}{(=)}$} $ \phi_f([[\Phi]](w))$ (i.e. strict preservation).
\end{condition}
\noindent The advantage of working with $\image(Z)$ is the ability to ignore $\s{T}$ and the stability condition.  This is purely linear algebra. We remark that although there is no explicit mention of the log branch in Condition 3, our subset $\image(Z)_{comp}$ is only meaningful for comparison when the log branch is at the positive real axis.

Condition \ref{cohomological_condition} may seem redundant after Condition \ref{order_preserving_condition}  and the discussion in \S{\ref{compatible_chapter}}, however, we needed the former to reliably define the latter.

\subsection{Proof that the criterion implies compatibility}
We now assume the autoequivalence $\Phi$ satisfies

\begin{enumerate}
\item[C1]\label{C1} $\coh{i}(\Phi(\sA)) = 0$ if $i \neq \{M, M-1\}$, $M \in \mZ$.
\item[C2]\label{C2} $[\Phi]$ descends to an automorphism $[[\Phi]]$ of $\image(Z)$.
\item[C3]\label{C3} Given $v,w \in \image(Z)_{eff \cap comp}$ with $\phi_f(v) $ {\tiny $\stackrel{<}{(=)}$} $\phi_f(w)$, then $\phi_f([[\Phi]](v))$ {\tiny $\stackrel{<}{(=)}$} $ \phi_f([[\Phi]](w))$ (i.e. strict preservation).
\end{enumerate}
See Definition \ref{faux_phase} or \S\ref{notation} for the definition of $\phi_f$.  Without loss of generality, we will assume $M = 0$ and fix the choice of log branch to be discontinuous on $\mR^{>0}$ and satisfying $\phi_f(A) \in (0,1]$ for $A \in \s{A}$.

We will repeatedly use the following proposition:

\begin{proposition}
\label{membership_please}
There exists an $m \in [0, 1]$ such that 
 \begin{enumerate}
  \item[(i)] $\Phi(\sP(0, m]) \subset \s{A}$
  \item[(ii)] $\Phi(\sP(m, 1]) \subset \s{A}[1]$
 \end{enumerate}
 \end{proposition}
\begin{remark} If $m = 0$ we take $\sP(0, 0]$ to mean the category consisting of only the zero object.
\end{remark}

\begin{remark}
By Proposition \cite[Lemma 1.1.2]{polishchuk}, Condition (C1) implies $\Phi(\sA)$ is a tilt of $\sA$.  However, $\sA^0$ and $\sA^1$ are not the tilting subcategories. 
\end{remark}

To prove the proposition, let $m  \in \mR$ be the minimal phase such that if $v \in \image(Z)_{eff \cap comp}$ then $\phi_f(v) \leq m + 1$. Clearly $m$ depends on both $\Phi$ and $\s{A}$. Condition (C3) ensures $m$ has the following important property:

\begin{lemma}
\label{psi_property}
If $F \in \s{A}$ is semistable and $\phi(F) \leq m$, then $Z([\Phi(F)]) \in \mathbb{H}^\prime$. Further, $m$ is maximal with respect to this property.
\end{lemma}
\begin{proof}
This is a consequence of the definitions and Lemma \ref{semistable_phases} below: by definition there exists a semistable $G$ with $[[G]] \in \image(Z)_{eff\cap comp}$,  $\phi(F) \leq \phi(G) \leq m_{\sA}$, and $Z([\Phi(G)]) \in \mathbb{H}^\prime$.  Lemma \ref{semistable_phases} then gives the lemma.
\end{proof}

\begin{lemma}
\label{semistable_phases}
Let $F, G \in \mr{Obj}(\s{A} \cup \s{A}[1])$ be semistable with $\coh{-2}(\Phi(G)) = 0$.  Then $\phi(F) $ {\tiny $\stackrel{<}{(=)}$} $ \phi(G)$ implies $\phi_f(\Phi(F)) $ {\tiny $\stackrel{<}{(=)}$} $ \phi_f(\Phi(G))$.
\end{lemma}
\begin{proof}
By definition, $[[F]], [[G]] \in \image(Z)_{eff}$.  The choice of log branch ensures the fake and real phases of $F$ and $G$ are equal.  Assume that $F, G \in \sA$.   Then $[[F]], [[G]] \in \image(Z)_{comp}$ and applying C3 gives the result.

In the case that $F \in \sA$ and $G \in \sA[1]$, we first note that  C1 and the assumption $\coh{-2}(\Phi(G)) = 0$ imply that $Z([G[-1]]) \in \mathbb{H}^\prime$ and $Z([\Phi]([G[-1]])) \in \mathbb{H}^\prime$.  This shows that $[[G]] \in \image(Z)_{comp}$.  Applying C3 gives the result.  

Lastly, let $F, G \in \sA[1]$.  We must show that $\phi(F) \leq \phi(G)$ ensures $[[F]] \in \image(Z)_{comp}$.  The case above applies to $F[-1]$ and $G[-1]$, showing $\phi_f(\Phi(F[-1])) \leq \phi_f(\Phi(G[-1])$.  Since $[[G]] \in \image(Z)_{comp}$, $\phi_f(\Phi(G)) \leq 1$.  This is enough to show that $Z([F[-1]]) \in \mathbb{H}^\prime$ and $Z([\Phi]([F[-1]])) \in \mathbb{H}^\prime$.  Therefore, $[[F]] \in \image(Z)_{comp}$ and applying C3  gives the result.
\end{proof}




Now that we have defined $m$ we will prove a series of lemmas. The first is adapted from \cite{Ruiperez_Lopez}.
\begin{lemma}
\label{region_pure}
Let $F \in \s{A}$,
\begin{enumerate}
\item if $\coh{i}(\Phi(F)) = 0$ for $i \neq -1$ then $m < \phi_f(F) \leq 1$.
\item if $\coh{i}(\Phi(F)) = 0$ for $i \neq 0$ then $0 < \phi_f(F) \leq m$.
\end{enumerate}
\end{lemma}

\begin{proof}
We will prove the first statement; the second is shown by similar methods.  Suppose the contrary: $\coh{i}(\Phi(F)) = 0$ for $i \neq -1$ yet  $0 < \phi_f(F) \leq m$.  The Harder-Narasimhan filtration of $F$ gives a short exact sequence
\begin{displaymath}
\shortexactseq{F_-}{F}{A_-(F)}
 \end{displaymath}
\noindent where $0 < \phi(A_{-}(F)) \leq m$.  The first inequality follows from the fact that $F \in \s{A}$.
Applying $\Phi$ gives a long exact sequence
\begin{align*}
\xymatrix@C=.95em{
\ldots \ar[r] & \coh{-1}(\Phi(F)) \ar[r] & \coh{-1}(\Phi(A_{-}(F))) \ar[r] &  \coh{0}(\Phi(F_{-}))\ar[r] &  & \ \qquad
} \\
\xymatrix@C=.95em{\ &\coh{0}(\Phi(F)) \ar[r] & \coh{0}(\Phi(A^{-}(F))) \ar[r] & 0}
\end{align*}
where the last $0$ is due to C1.

By Lemma \ref{psi_property}, we have $Z([A_{-}(F)]) \in \mathbb{H}^\prime$ and a semistable factor $A^\pr = A_{-}(\Phi(A_{-}(F)))$ with $\phi(A^\pr) \leq 1$ guaranteeing that $\coh{0}(\Phi(A_{-}(F))) \neq 0$.  The exactness of the above sequence then shows $\coh{0}(\Phi(F)) \neq 0$,  giving us our desired contradiction.
\end{proof}

For an object $F \in \s{A}$, the condition $0 < \phi_f(F) \leq m$ is coarser than $F \in \s{P}(0, m]$.  This allows us to work with objects without assuming membership in $\sP(0, m]$ or $\sP(m, 1]$.

\begin{lemma}[\cite{Ruiperez_Lopez}]
\label{exact_sequence}
Let $F \in \s{A}$.  Then there exists a short exact sequence
 \begin{displaymath}
\shortexactseq{\Phi^{-1}(\coh{-1}(\Phi(F))[1])}{F}{\Phi^{-1}(\coh{0}(\Phi(F)))}
\end{displaymath}
of objects in $\s{A}$.
\end{lemma}
\begin{proof}
We will include our own proof.  Let $G := \Phi(F)$.  We have the triangle
\begin{displaymath}
\xymatrix@C=.2em@R=.9em{\coh{-1}(G)[1]   & \ar[rr]&&& G \ar[dll] \\
		&& \coh{0}(G) \ar@{-->}[ull] &&}
\end{displaymath}
coming from the t-structure.  Applying $\Phi^{-1}$ results in the long exact sequence:
\begin{align*}
\xymatrix@C=.9em{
\ldots \ar[r] & \coh{-1}(\Phi^{-1}(\coh{0}(G))) \ar[r] & \coh{0}(\Phi^{-1}(\coh{-1}(G)[1])) \ar[r] &  F \ar[r] &  & \ \qquad
} \\
\xymatrix@C=.9em{\ &\coh{0}(\Phi^{-1}(\coh{0}(G))) \ar[r] & \coh{1}(\Phi^{-1}(\coh{-1}(G)[1])) \ar[r] & \ldots}
\end{align*}
C1 ensures $\coh{k}(\Phi^{-1}(\coh{-1}(G)[1])) = 0$ for $k \neq 0, -1$ and $\coh{k}(\Phi^{-1}(\coh{0}(G))) = 0$ for $k \neq 0, 1$.  Thus the two end groups in this sequence vanish.

We are left showing
\begin{align*}
 \coh{0}(\Phi^{-1}(\coh{-1}(G)[1])) \cong \Phi^{-1}(\coh{-1}(G)[1]) \\
\coh{0}(\Phi^{-1}(\coh{0}(G))) \cong \Phi^{-1}(\coh{0}(\Phi(F))).
\end{align*}
The first isomorphism is a consequence of $\coh{-1}(\Phi^{-1}(\coh{-1}(G)[1])) = 0$, which can be seen from further up on the same sequence:
\begin{displaymath}
 \xymatrix@C=.9em{
 \coh{-2}(\Phi^{-1}(\coh{0}(G)) \ar[r] & \coh{-1}(\Phi^{-1}(\coh{-1}(G)[1]))  \ar[r] & \coh{-1}(F)
 }
\end{displaymath}
Clearly the left and right groups are zero, thus giving the result.  A similar argument holds for the case of $\coh{0}(\Phi^{-1}(\coh{0}(G))) \cong \Phi^{-1}(\coh{0}(\Phi(F)))$, concluding the proof.
\end{proof}

We will now analyze how $\Phi$ acts on semistable objects of $\s{A}$.
\begin{lemma}
\label{preserves_sheaves}
Let $F \in \s{A}$ be semistable. Then, $\Phi(F)$ is cohomologically pure, i.e.,  $\coh{i}(\Phi(F)) = 0$ for all but one integer.
\end{lemma}
\begin{proof}
We separate the proof into two cases.

$0 < \phi(F) \leq m$:  Suppose that both $\coh{0}(\Phi(F)) \neq 0$ and $\coh{-1}(\Phi(F)) \neq 0$.  The exact sequence in Lemma \ref{exact_sequence}  shows that $\Phi^{-1}(\coh{-1}(\Phi(F))[1])$ is a subsheaf of $F$ (and non-zero by assumption). $F$ being semistable, combined with its assumed phase ensures $\phi_f(\Phi^{-1}(\coh{-1}(\Phi(F))[1])) \leq m$.  However, $\Phi^{-1}(\coh{-1}(\Phi(F))[1])$ satisfies Lemma \ref{region_pure}(1),  showing $\phi_f(\Phi^{-1}(\coh{-1}(\Phi(F))[1])) > m$, a contradiction.

$m < \phi(A) \leq 1$:  Again, suppose the contrary.  Then $\Phi^{-1}(\coh{0}(\Phi(F)))$ is a quotient sheaf of $F$, showing $\phi_f(\Phi^{-1}(\coh{0}(\Phi(F)))) > m$.  However, $\Phi^{-1}(\coh{0}(\Phi(F)))$ satisfies Lemma \ref{region_pure}(2), giving  $\phi_f(\Phi^{-1}(\coh{0}(\Phi(F)))) \leq m$, a contradiction.
\end{proof}
Proposition \ref{membership_please} is an easy consequence of Lemma \ref{preserves_sheaves} since it is generated from semistable elements through extension.
\begin{proof}[Proof of Proposition \ref{membership_please}]
Let $\sA^0 = \sP(0,m]$ and $\sA^1 := \sP(m, 1]$.  Assume $F \in \s{A}^{i}, i \in \{0, 1\}$.  Let $A_j(F)$ denote the semistable objects in the Harder-Narasimhan filtration of $F$.   By definition of $\s{A}^i$,  $A_j(F) \in \s{A}^i$.  Lemma \ref{preserves_sheaves} and Lemma \ref{region_pure} imply that $\Phi(A_j(F))$ is cohomologically pure and are concentrated in degree $-i$ for all $j$ (since $-0$ is $0$).  Since $\s{A}$ and $\s{A}[-1]$ are extension closed,  we have that $\Phi(F)$ is concentrated in degree $-i$, completing the proof.
\end{proof}

From Lemma \ref{preserves_sheaves}, we know that $\Phi(F)$ is cohomologically pure. However, this t-structure was arbitrary in the family of t-structures generated by $\stab$.  Since semistable objects are the objects that are cohomologically pure in \textit{every} t-structure generated from $\stab$, one can show that $\Phi(F)$ is semistable if $\Phi$ satisfies our conditions for all generated t-structures.  This motivates the proof of our main theorem.  Before stating the next theorem, we recall the definition of two key subsets of $\image(Z)$:
\begin{displaymath}
\image(Z)_{eff} =  \{\pm v \;| \;\exists \textrm{ s.s. } G \in \s{A} \textrm{ with } Z([G]) = v\}.
\end{displaymath}
\begin{displaymath}
\image(Z)_{comp} =  \{ v | \; Z(v) \in \mathbb{H}^\prime \}  \cup \{\pm v |\;  Z(v) \in \mathbb{H}^\prime \textrm{ and } Z([[\Phi]]v) \in \mathbb{H}^\prime\}
\end{displaymath}

\begin{theorem}
\label{stable_preserving}
Let $\Phi \in \mr{Aut}(\s{T})$.  Given $(Z, \s{P}) \in \mr{Stab}(\s{T})$, a locally finite stability condition, one has the t-structure $\s{D}^{\leq 0} = \s{P}(0, \infty)$, with heart $\s{A}$.
If $\Phi$ satisfies
\begin{enumerate}
\item[(i)] $\Phi(\s{A}) \subset \s{D}^{\leq M} \cap \s{D}^{\geq M -1}$, $M \in \mZ$.
\item[(ii)] $\Phi$ descends to an automorphism $[[\Phi]]$ of $\image(Z) = K(\s{T})/\ker(Z)$.
\item[(iii)]  for $v,w \in \image(Z)_{eff \cap comp}$, the relation $\phi_f(v)$ {\tiny $\stackrel{<}{(=)}$} $\phi_f(w)$ implies $\phi_f([[\Phi]](v))$ {\tiny $\stackrel{<}{(=)}$} $\phi_f([[\Phi(G)]](w))$; here $\phi_f$ is the ``implied phase" obtained from a log branch.
\end{enumerate}
then $\Phi$ is compatible with $\stab$.
\end{theorem}
\begin{proof}
Let $(Z_{\theta}, \s{P}_{\theta})$ be the stability condition $(Z, \sP)\cdot e^{i\pi\theta}$, with heart $A_{\theta} \cong \s{P}(\theta, \theta + 1] = \sP_{\theta}(0, 1]$.  We use the notation $M_\theta$, $m_{\theta}$, $\phi_f^\theta$, and $\image(Z)^\theta_{comp}$ for the obvious objects associated to $(Z_\theta, \s{P}_\theta)$.  Note that generally,  $\image(Z)^\theta_{comp}$ will not be equal to $\image(Z)^0_{comp}$: the log branch will change, thus changing which elements of $\image(Z)$ are ``comparable".  We claim if $\Phi$ satisfies (i), (ii), and (iii) above for $(Z, \sP)$, then it will also satisfy them for $(Z_{\theta}, \s{P}_{\theta})$, $\theta \in \mR$. Only (i) and (iii) are unclear.  Once this is shown, it is easy to see that $\Phi$ is semistable preserving: if not, we can choose $\theta$  to ensure $\Phi(F)$ is not cohomologically pure (in the t-structure with heart $\s{A}_\theta$).  The full compatibility easily follows.

To show that $\Phi$ satisfies (i), (ii), and (iii) with regard to $(Z_\theta, \s{P}_\theta)$ it is enough to assume that $\theta \in (0, 1]$, for the shift [1] is strongly compatible.  We split the argument into the cases $\theta < m_0$ and $\theta > m_0$.  The main difficulty with the latter is showing (i) is satisfied:  one encounters problems understanding $\Phi(\sP(m_0, \theta][1])$.  We first handle $\theta \leq m_{0}$.

By definition, $\s{A}_{\theta}$ is the extension-closed full subcategory of $\sT$ generated by $\s{P}(\theta, m_0]$, $\s{P}(m_0, 1]$, and $\s{P}(0, \theta)[1]$.  Let $m^\prime_\theta \in \mR$ be the maximal number such that for semistable $F \in \s{P}(\theta, 1]$,  $\phi_f(\Phi(F)) > m^\prime_\theta$ (it exists for the same reason $m_\theta$ exists).  Our choice of $\theta$ and $M_0$ ensures $0 < m^\prime_\theta  \leq 1$.  If $m^\prime_\theta \geq \theta$, then by Proposition \ref{membership_please}, and condition (iii)
\begin{align*}
\Phi(\s{P}(\theta, m_0]) & \subset \s{P}(\theta, 1] \subset  \s{A}_{\theta}, \\
\Phi(\s{P}(0, \theta)[1]) &\subset \s{P}(1, 2] \subset \s{D}_{\theta}^{\leq 0} \cap \s{D}_{\theta}^{\geq -1}, \\
\Phi(\s{P}(m_0, 1]) &\subset \s{P}(1, 2] \subset \s{D}_{\theta}^{\leq 0} \cap \s{D}_{\theta}^{\geq -1}.
\end{align*}
Thus showing that $\Phi$ satisfies (i) in $(Z_\theta, \sP_\theta)$ with $M_\theta = 0$.  On the other hand, if $m_\theta^\prime < \theta$, then
\begin{align*}
\Phi(\s{P}(\theta, \psi]) & \subset \s{P}(\theta -1, 1] \subset  \s{D}_{\theta}^{\leq 1} \cap \s{D}_{\theta}^{\geq 0}, \\
\Phi(\s{P}(0, \theta)[1]) &\subset \s{P}(1, m_{\theta}^\prime] \subset \s{D}_{\theta}^{\leq 1} \cap \s{D}_{\theta}^{\geq 0}, \\
\Phi(\s{P}(m_0, 1]) &\subset \s{P}(1, m_\theta^\prime + 1] \subset \s{D}_{\theta}^{\leq 1} \cap \s{D}_{\theta}^{\geq 0}.
\end{align*}
Again showing that $\Phi$ satisfies (i) ($M_\theta = 1$) with respect to $(Z_\theta,\sP_\theta)$.

To show (iii) is satisfied; we only need to check the case of strict inequality since equality is clearly satisfied.  Any choice of log branch splits $\mC^{\ast}$ into two symmetric pieces.  $\mathbb{H}^{\prime}$ is adapted to the branch with discontinuity at the positive real axis. By definition, $\image(Z)^{\theta}_{comp}$ is defined using $\mathbb{H}^{\prime} \cdot e^{i\pi\theta}$ and a log branch with discontinuity at the ray of angle $\pi\theta$ (relative to the positive real axis).  We can assume that on the positive real axis, the phases of the two branches agree.  With this description, we can rephrase (iii) for $\image(Z)_{eff} \cap \image(Z)^{\theta}_{comp}$ in terms of $\image(Z)_{eff} \cap \image(Z)^{0}_{comp}$.

Let $m_\theta^\prime \geq \theta$.  For $v,w \in \image(Z)_{eff} \cap \image(Z)^{\theta}_{comp} \cap \image(Z)^{0}_{comp}$ (iii) is clear: $v, w, [[\Phi]]v$, and $[[\Phi]]w$ will have the same phase assignment in both branches.  For the case $w \notin \image(Z)^{0}_{comp}$,  $Z([[\Phi]]w) \in \mathbb{H}^\prime$.   Our choice of log branch for $\phi_f^\theta$ ensures $\phi^\theta_f([[\Phi]]w) > 2$.  Since $v \in \image(Z)_{comp}^0$, $\phi^\theta_f([[\Phi]]v) \leq 2$, thus handling this case.  To handle $v,w \notin \image(Z)^{0}_{comp}$ one first notes Lemma \ref{semistable_phases} implies $m_0 + 1 < \phi_f(v)$.  Therefore, $-v, -w \in \image(Z)^{0}_{comp} \cap \image(Z)^{\theta}_{comp} $.  The linearity of $[[\Phi]]$, and the above case then gives the result.   To handle $m_\theta^\prime < \theta$ it suffices to prove (iii) for $[1] \circ \Phi$.  For this autoequivalence, identical arguments as above can be used.

We have now shown $\Phi$ satisfies (i), (ii), and (iii) with regard to $(Z_{\theta}, \s{P}_{\theta})$ for $\theta \leq m_0$. 
Consider the sequence $\omega_1 = m_0$, $\omega_2 = m_{\omega_1}$, etc.  The sequence has the property that $\Phi$ satisfies the conditions stated above for each $(Z_{\omega_i}, \sP_{\omega_i})$.   Note if the conditions are met for $(Z_{\omega}, \sP_{\omega})$ with $\theta - \omega \leq m_{\omega}$, then the above arguments ensure that it is true for $\theta$.  This implies if  $\lim_{i \to \infty} \omega_i =  \infty$, then we are done.  Thus to finish our claim that $\Phi$ satisfies (i), (ii), (iii) for all $(Z_{\theta}, \sP_{\theta})$ it suffices to show if $\lim_{i \to \infty} \omega_i = \kappa < \infty$ then $\Phi$ satisfies the condition for $(Z_{\kappa + \rho}, \sP_{\kappa + \rho})$ for $\rho \in [0, 1)$.  



We first show show that $\Phi(\sP[\kappa, \kappa + 1)) \subset \sP[\kappa + 1, \kappa + 2)$.  Let $F \in\sP[\kappa, \kappa + 1)$ be a semistable object.  Then for some $j$ with $\psi_j$ small (or possibly zero), $F \in \s{A}_{\omega_j}$ and $\phi(F) > m_{\omega_j}$.  By Proposition \ref{membership_please}, $\Phi(F) \in \s{A}_{\omega_j}[1]$.  If $\Phi(F)\notin \sP[\kappa, \kappa + 1)[1]$, then there exists at least one semistable factor of $\Phi(F)$ not contained $\s{B}_{\kappa}[1]$.  The definition of $\kappa$  implies that we must have a $k > j$ such that $F \in \s{A}_{\omega_k}$, yet $\Phi(F) \notin \s{A}_{\omega_k}[1]$.  Using Lemma \ref{preserves_sheaves} again,  $\Phi(F) \in \s{A}_{\omega_k}$ resulting in $\phi(F) < m_{\omega_k}$, a contradiction. 

The autoequivalence $[-1] \circ \Phi$ restricts to an element in $\Aut(\sP[\kappa, \kappa + 1))$.  Clearly this is enough to show $[-1] \circ \Phi$ satisfies (i) with respect to $(Z_\kappa, \sP_\kappa)$.  For (iii), note that given $v, w \in \image(Z)^{\kappa}_{comp}$, there exists $\omega_l$ such that $v, w \in \image(Z)^{\omega_l}_{comp}$.  The condition is assumed true for $v, w \in \image(Z)^{\omega_l}_{comp}$.  We can choose our log branches for $\omega_l$ and $\kappa$ in a compatible manner such that $\phi_f^{\omega_l}$ and $\phi_f^{\kappa}$ agree on $v, w, [[\Phi]]v$ and $[[\Phi]]w$, thus showing that (iii) is satisfied.   

We can therefore apply the above machinery to $\rho \leq m_{\kappa}$  (with respect to $[-1] \circ \Phi)$.  However, since $[-1] \circ \Phi(\sP[\kappa, \kappa + 1)) \subset \sP[\kappa, \kappa +1)$, it is clear that  if $m_{\omega_k} \neq 1$ and $\rho > m_\kappa$ then $(Z_{\kappa + \rho}, \sP_{\kappa + \rho})$ and $(Z_{\kappa + m_\kappa}, \sP_{\kappa + m_\kappa})$  differ only by a shift of ``fake" phases.  This proves our claim.

We will now show that $\Phi$ is compatible with $(Z, \s{P})$.  Given $F$ semistable in $(Z, \s{P})$ with $\phi(F) = \eta$, suppose that $\Phi(F)$ is not semistable.  The Harder-Narasimhan filtration of $\Phi(F)$ gives $\phi_{-}(\Phi(A)) < \phi_f(\Phi(F) <  \phi_0(\Phi(A))$.  If we set $\theta = \phi_f(\Phi(F)$, in the stability condition $(Z_{\theta}, \s{P}_{\theta})$, $\coh{0}(\Phi(F)) \neq 0$ and $\coh{1}(\Phi(F)) \neq 0$.  From above, we know that $\Phi$ satisfies conditions (i), (ii), and (iii) for $(Z_{\theta}, \s{P}_{\theta})$, allowing us to apply Lemma \ref{preserves_sheaves}, and get a contradiction.  Therefore $\Phi(F)$ is semistable in $(Z, \s{P})$, with phase $\theta$.  We can therefore remove the ``f" from $\phi_f$.  Condition (iii) is then the condition for compatibility.
\end{proof}

\begin{corollary}
\label{still_compatible}
If $\Phi$ is compatible with $(Z, \sP)$ then it is compatible with $(Z, \sP)\cdot g$ for $g \in \widetilde{\mr{GL}^{+}}(2, \mR)$.
\end{corollary}

\begin{corollary}
Let $\Phi$ satisfy all conditions of Theorem \ref{stable_preserving}.  Then $\Phi$ is strongly compatible if and only if the induced automorphism on $\image(Z)$ extends to an orientation preserving $\mR$-linear automorphism of $\mC$.
\end{corollary}

In the case that $K(\s{T})$ is finite rank, Theorem \ref{stable_preserving}(ii) can be rephrased.

\begin{proposition}
$\Phi$ descends to an automorphism of $\image(Z)$ if and only if $\Phi$ restricts to an automorphism of $\ker Z$.  Further, if $\rank(\ker Z)$ is finite, this condition is equivalent to $\Phi(\ker Z) \subseteq \ker Z$.
\end{proposition}
\begin{proof}
The first statement is a consequence of the triangle axioms in the derived category of $\mZ$ modules.  For the second, it is clearly a necessary condition.  For sufficiency, first note that the map is injective, so we just need to show surjectivity.
If not surjective, there exists an element $a \in \ker(Z) \notin \Phi(\ker{Z})$, thus $\Phi^{-1}(\ker Z) \nsubseteq \ker Z$.  Our assumption can be rephrased as $\ker Z \subset \Phi^{-1}(\ker Z)$.  The equivalence of rank (since both are finite rank) implies that $\Phi^{-1}(\ker Z)/\ker Z$ is a torsion group.  However, $\Phi^{-1}(\ker Z)/\ker Z \subset \image Z \cong \image Z \subset \mathbb{C}$ which has no torsion points, a contradiction.  Thus $\Phi$ is an automorphism when restricted to $\ker(Z)$.
\end{proof}

\section{n-gons}
\label{n-gons_chapter}
This chapter is a collection of facts and definitions about n-gons and their derived categories.  We define the ``classical" stability condition, which will be used throughout the rest of this paper.  Throughout the rest of the paper, our base field will be $\mC$.  When no confusion can arise, we omit the structure sheaf in the tensor notation.

\subsection{The geometry of n-gons.}
Let $\E_n$, $n \in \mathbb{N}$, denote the n-gon:  projective singular reducible curves consisting of a cycle of n components, all isomorphic to $\mP$, with nodal singularities (i.e. transverse intersections).  $\E_m$ is a Galois cover of $\E_n$ if and only if $n|m$.  In particular, n-gons are Galois covers of the Weierstrass nodal cubic.  We fix a consistent choice of morphisms $\{\pi_{m,n} \in \phom(\E_m, \E_n)\}_{m,n \in \mZ, n|m}$ and  $\{\iota_{m,n} \in \mr{Gal}(\E_m, \E_n)\}_{m,n \in \mZ, n|m}$ satisfying $\pi_{m, n} \circ \pi_{l, m} = \pi_{l,n}$ and $\pi_{m,n} \iota_{m,l}  = \iota_{n,l}$ for $l | n | m$ and $\iota_{n,1}$ corresponds to the first root of unity under the natural isomorphism $\mr{Gal}(\E_m, \E_n) \cong \mu_n$.   We oftentimes omit the second number when it is 1 e.g.  $\pi_{m}$ for $\pi_{m,1}$.

The normalization of $\E_n$ is $\amalg_{n} \mathbb{P}^1$.  
 We label the components with $\mZ/n\mZ$ in manner  consistent with the action of our choice of $\iota_n$.

We will also need the projective genus 0 singular curves $\I_m$.  These curves are a chain of  $m$ reducible components, all isomorphic to $\mP$.  They can be obtained as partial normalizations of the $m$-gon at any one of its singular points.  Using these curves, we can alternatively characterize $\iota_{m,n}$ as a consistent choice of elements in $\mr{Gal}(\E_m, \E_n)$ such that $\iota_{m,n} \I_n \cap \I_n \neq \emptyset$ for any $\I_n \subset \E_m$.

\begin{figure}[htb!]
\centering
\includegraphics{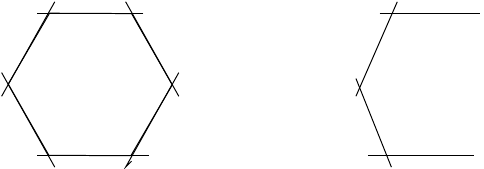}
\caption{$\E_6$ and $\I_4$}
\end{figure}

\subsection{$K(\bderive{\E_n})$ and $\sigma_{cl}(n)$}
The isomorphism $K(\bderive{\E_n}) \cong \mZ^{n+1}$ is demonstrated by analyzing the short exact sequences supplied by the adjunction map $\mr{Id} \rightarrow \eta_{\ast}\eta^{\ast}$, see \cite[Proposition 2.3]{Ruiperez_Lopez}.  Our preferred basis for $K(\bderive{\E_n})$ consists of $e_{0} = [k(p)]$ for a smooth point $p \in \E_n$, and $e_{i} = [\eta_\ast(\s{O}_{\mP_{i}}(-1))]$, $0 < i \leq n$.  Define $\mr{rank}_{i}(\sF)$ as the dimension of the vector space obtained by restricting $\sF$ to the generic point of the $i$th component.  Clearly $[\sF] = \chi(\sF)e_0 + \Sigma_{n}\mr{rank}_i(\sF)e_i$.

Let  $\sigma_{cl}(n) = (Z_{cl}, \s{P}_{cl})$ denote the stability condition with charge $Z_{cl}(F) = -\chi(F) + i\rktot(F)$, with heart $\mr{Coh}(\E_n)$.  Here $\rktot$ designates the function $\sum_{0 < i \leq n} \mr{rank}_{i}$.  It is clear that $\ker(Z_{cl})$ is finite rank and $\image(Z_{cl})$ is rank 2.

This stability condition is an extension of classical slope to the case of n-gons.  Given a torsion free $\sF \in \Coh(\E_n)$, one can define its slope $\psi(\sF) := \frac{\chi(\sF)}{\rktot(\sF)}$.   In the case that $\sF$ is semistable (in $\sigma_{cl}(n)$), the conversion between $\psi(\sF)$, its slope, and $\phi(\sF)$, its phase, is given by
\begin{displaymath}
\psi(\sF) = - \cot(\pi \phi(\sF)).
\end{displaymath}

\subsection{Classification of torsion free sheaves.}

\begin{theorem}[\cite{MR2264108} Theorem 1.3]
\label{tf_bundles}
With $\E_n$ and $\mathbb{I}_k$ as above, let $\sE$ be an indecomposable torsion free sheaf on $\E_n$.
\begin{enumerate}
\item If $\s{E}$ is locally free, then there is an $\acute{e}tale$ covering $\pi_{nr,n}: \E_{nr} \rightarrow \E_n$, a line bundle $ \s{L} \in Pic(\E_{nr})$, and a natural number $m \in \mathbb{N}$ such that
\begin{displaymath}
\s{E} \cong \pi_{nr,r\ast}(\s{L}\otimes\s{F}_m),
\end{displaymath}
where $\s{F}_{m}$ is indecomposable and locally free on $\E_{nr}$, recursively defined by the sequences
\begin{displaymath}
\begin{array}{lll}
\begin{CD}
 0 @>>> \s{F}_{m-1} @>>> \s{F}_m @>>> \s{O}_{\E_{nr}} @>>> 0
\end{CD}, & m \geq 2, & \s{F}_{1} = \s{O}_{\E_{nr}}
\end{array}.
\end{displaymath}
\item If $\sE$ is not locally free then there exists a finite map $p_k: \mathbb{I}_k \rightarrow \E_n$ and a line bundle $\sL \in Pic(\mathbb{I}_k)$ (where $k, p_k$ and $\sL$ are determined by $\sE$) such that $\sE \cong p_{k\ast}(\sL)$. 
\end{enumerate}
\end{theorem}

\begin{example}
\label{sph_vb}
If $\sV \cong  \pi_{nr,r\ast}(\s{L}\otimes\s{F}_m)$ with $m > 1$, then $\dim \End(\sV) > 1$. 
\end{example}

\subsection{Automorphism and Picard groups of $\E_n$}

There are natural isomorphisms $\Aut(\E_1) \cong \mC^\ast \rtimes Z_2$ and $\Pic^0(\E_1) \cong \mathbb{G}_m \cong \mC^\ast$.  For any closed $\lambda \in \mC^\ast$, we will let $t_\lambda \in \Aut(\E_1)$ denote the automorphism that restricts to translation by $\lambda$ on the smooth locus (it fixes the singular point) and $\sO_{\E_1}(0; \lambda)$ will denote the invertible sheaf $\sO_{\E_1}(t_{\lambda_1}(p) - p)$ for a choice of smooth closed $p \in \E_1$.  It is independent of the choice of $p$ since if $\sL \in \Pic^0(X)$ then $t_{\lambda_1}^\ast \sL \cong \sL$.  This behavior mirrors similar statements on elliptic curves. Note that the invariance implies $\sO_{\E_1}(0; \lambda_1) \otimes \sO_{\E_1}(0; \lambda_2) \cong \sO_{\E_1}(0; \lambda_1 \lambda_2)$. 

\subsection{ $\Coh(\E_n)$ as $\pi_{n\ast}\sO_{\E_n}$-modules.}
\label{littlemod}
The maps $\pi_{n}$ are affine, thus $\pi_{n\ast}\sO_{\E_n}$ has a natural monoid structure in $\Coh(\E_1)$ (with the standard tensor monoidal structure).  That is to say,  there exists an associative morphism $\pi_{n\ast}\sO_{\E_n} \otimes \pi_{n\ast}\sO_{\E_n} \xrightarrow{B} \pi_{n\ast}\sO_{\E_n}$.  As carried out in \cite{EGAII}, one has an equivalence $\Coh(\E_n) \cong \pi_{n\ast}\sO_{\E_n}$-mod.  The purpose of this section is to give an explicit description of the ring structure on $\pi_{n\ast}\sO_{\E_n}$ and the action of $\iota_{n\ast}^k$ on $\pi_{n\ast}\sO_{\E_n}$-mod.  Throughout this section $\xi$ is a primitive $n$th root of unity.  

Over $\mC$,  $\pi_{n\ast}\sO_{\E_n} \cong \oplus_{0 \leq j < n} \sO_{\E_1}(0; \xi^k)$. This decomposition, combined with the calculation
\begin{displaymath}
\phom(\sO_{\E_1}(0; \lambda_1), \sO_{\E_1}(0; \lambda_2)) = \begin{cases} 0 & \lambda_1 \neq \lambda_2 \\ \mC & \lambda_1 = \lambda_2 \end{cases}
\end{displaymath}
 and isomorphism $\sO_{\E_1}(0; \xi^j)\otimes \sO_{\E_1}(0; \xi^k) \cong \sO_{\E_1}(0; \xi^{j + k})$, implies $B$ can be described by a choice of isomorphism:
\begin{displaymath}
B_{jk}:   \sO_{\E_1}(0; \xi^j) \otimes \sO_{\E_1}(0; \xi^k) \rightarrow \sO_{\E_1}(0; \xi^{j + k})
\end{displaymath}
This choice satisfies a strong consistency condition: setting $B_{i_0, i_1, \ldots, i_s} = B_{i_0, \Sigma_{0 < j \leq s} i_{j}}\circ \ldots \circ B_{i_{s-2}, i_{s-1} + i_{s}} \circ B_{i_{s-1}, i_s}$, then $B_{i_0, i_1, \ldots, i_s}$ is independent of composition order. Therefore, $B$ consists of a series of structure maps: 
\begin{displaymath}
  B_{i_0, i_1, \ldots, i_s}: \sO_{\E_1}(\sO; \xi^{i_0}) \otimes \sO_{\E_1}(0; \xi^{i_1}) \otimes \ldots \otimes \sO_{\E_1}(0; \xi^{i_s}) \xrightarrow{\cong} \sO_{\E_1}(0; \xi^{\Sigma_{0 \leq j < s} i_j})
\end{displaymath}
for $s \in \mN$ satisfying relations of the form $B_{i, j+l} \circ B_{j, l} = B_{i, j, l} = B_{i + j, l}\circ B_{i, j}$, $i, j, k \in \mZ_n$ (and many others).  

The data $\iota_n^k: \E_n \to \E_n$ (specifically, the isomorphism $\sO_{\E_n} \to \iota_{n\ast}^k \sO_{\E_n}$) translates to a monoid isomorphism $\hat \iota_n^k: \pi_{n\ast}\sO_{\E_n} \to \pi_{n\ast}\sO_{\E_n}$ (since $\pi_n \circ \iota_n \cong \pi_n$).  The above fact about $\mr{Hom}$-spaces ensures $\hat \iota_n^k|_{\sO_{\E_1}(0; \xi^l)} = \xi^{kl} \cdot \mr{Id}$, i.e., it acts by the scalar $\xi^{kj}$ on the $j$th component.  Given $\sF \in \Coh(\E_n)$ and $M_{\sF}$ a choice of corresponding $\pi_{n\ast}(\sO_{\E_n})$-module with structure morphism $\pi_{n\ast}(\sO_{\E_n}) \otimes M_{\sF} \xrightarrow{m} M_{\sF}$, $M_{\iota_{n\ast}^k \sF} = M_{\sF}$ (as objects in $\Coh(\E_1)$) with structure morphism $\hat m$ given by $\hat m|_{\sO_{\E_1}(0; \xi^j)\otimes M} = \xi^{kj} m|_{\sO_{\E_1}(0; \xi^j)\otimes M}$.

\subsection{ $\Coh(\E_n \times \E_n)$ as $(\pi_n \times \pi_n)_\ast \sO_{\E_n \times \E_n}$-modules}
\label{mod_structure}
For $\E_n\times \E_n$ and $\E_1 \times \E_n$, much of the statements in the previous section apply.  Mainly, $\pi_n \times \pi_n$ and $\mr{Id} \times \pi_n$ are flat and affine.  Thus, we know $\Coh(\E_n \times \E_n) \cong (\pi_n \times \pi_n)_\ast \sO_{\E_n \times \E_n}$-mod (similarly for $\Coh(\E_1 \times \E_n)$).  Our goal is to give explicit description of the ring structure on $(\pi_n \times \pi_n)_\ast \sO_{\E_n \times \E_n}$ and the action of $\iota^j_n \times \iota_n^k$.  First, we decompose $(\pi_n \times \pi_n)_\ast \sO_{\E_n \times \E_n}$.  To simplify notation, when convenient we set $\wtpi^{nm}_{kl} := \pi_{n, k} \times \pi_{m, l}$, if $n = k$ then $\wtpi_{nl}^{nm} := \mr{Id} \times \pi_{m, l}$ (likewise for $m = l$).  Let $\rho^i$ be the projection of $\E_1 \times \E_1$ onto the $i$th factor. 

 \begin{lemma}\mbox{}
\label{notouch}
   \begin{enumerate}
   \item   $(\pi_n \times  \pi_n)_\ast \sO_{\E_n \times \E_n} \cong \oplus_{0 \leq i, j < n} \rho^{1\ast}\sO_{\E_1}(0; \xi^i) \otimes_{\sO_{\E_1 \times \E_1}} \rho^{2\ast} \sO_{\E_1}(0; \xi^j)$. 
    \item  $\mr{Hom}_{\E_1 \times \E_1}(\rho^{1\ast} \sO_{\E_1}(0; \xi^k), \rho^{2\ast} \sO_{\E_1}(0; \xi^j)) \cong \begin{cases} 0 & k, j \neq 0 \\ \mC & k = j = 0 \end{cases}$.
    \end{enumerate}
\end{lemma}
\begin{proof}
To begin, we use the following commutative diagram to set notation.  The two squares are Cartesian. 
\begin{displaymath}
  \xymatrix{ \E_n \ar[d]^{\pi_{n}} & \E_n \times \E_n \ar[d]^{\wtpi_{1n}^{nn}} \ar[l]^{\rho_{nn}^{1}}&  \\ \E_1 & \ar[l]_{\rho_{1n}^1} \E_1 \times \E_n \ar[d]^{\wtpi_{11}^{1n}} \ar[r]^{\rho_{1n}^{2}} & \E_n \ar[d]^{\pi_{n}} \\ & \E_1 \times \E_1 \ar[lu]^{\rho_{11}^1} \ar[r]^{\rho_{11}^{2}} & \E_1}
\end{displaymath}
Using proper base change,
\begin{align*}
  \wtpi^{n n}_{1n\ast} \sO_{\E_n \times \E_n} \cong \wtpi^{nn}_{1n\ast} \rho_{nn}^{1\ast} \sO_{\E_n} \cong \rho_{1n}^{1\ast} \pi_{n\ast}\sO_{\E_n} \cong \oplus_{0 \leq k < n} \rho_{1n}^{1\ast} \sO_{\E_1}(0; \xi^k).
\end{align*}
However,
\begin{align*}
  \widetilde \pi^{1n}_{11\ast} \rho_{1 n}^{1\ast} \sO_{\E_1}(0; \xi^k) & \cong \widetilde \pi^{1n}_{11\ast} \wtpi^{1n\ast}_{1 1} \rho_{11}^{1*} \sO_{\E_1}(0; \xi^k) \\
& \cong  \wtpi^{1n}_{11\ast} \sO_{\E_1 \times \E_n} \otimes_{\sO_{\E_1 \times \E_1}} \rho_{11}^{1*} \sO_{\E_1}(0; \xi^k) \\
& \cong \widetilde \pi^{1n}_{1,1\ast} (\rho_{1n}^{2\ast} \sO_{\E_n}) \otimes_{\sO_{\E_1 \times \E_1}} \rho_{11}^{1*} \sO_{\E_1}(0; \xi^k)\\
& \cong (\rho_{11}^{2\ast} \pi_{n\ast} \sO_{\E_n}) \otimes_{\sO_{\E_1 \times \E_1}} \rho_{11}^{1\ast} \sO_{\E_1}(0; \xi^k) \\
& \cong \oplus_{0 \leq i < n} \rho_{11}^{2\ast} \sO_{\E_1}(0 ; \xi^i) \otimes_{\sO_{\E_1 \times \E_1}} \rho_{11}^{1\ast} \sO_{\E_1}(0; \xi^k)
\end{align*}
Summing over these equivalences gives the first isomorphism.

For the second, we note
\begin{align*}
 \mr{Hom}_{\E_1 \times \E_1}(\rho^{1\ast} \sO_{\E_1}(0; \xi^k), \rho^{2\ast} \sO_{\E_1}(0; \xi^j)) & \cong \mr{Hom}_{\E_1}(\sO_{\E_1}(0; \xi^k), R\rho_\ast^{1}\rho^{2\ast} \sO_{\E_1}(0; \xi^j)) \\
 & \cong \mr{Hom}_{\E_1}(\sO_{\E_1}(0; \xi^k), \Phi_{\sO_{\E_1 \times \E_1}}(\sO_{\E_1}(0; \xi^j))\\
 \end{align*}
Here, the symmetry of $\sO_{\E_1 \times \E_1}$ allows us to reverse the direction of the transform.  We have a triangle in $\bderive{\E_1}:$
 \begin{align*}
  \Phi_{\sO_{\E_1 \times \E_1}}(\sO_{\E_1}(0; \xi^j)) \to \Phi_{\sO_{\Delta}}(\sO_{\E_1}(0; \xi^j)) \to \Phi_{\s{I}_\Delta[1]}(\sO_{\E_1}(0; \xi^j)) \to   
 \end{align*}
Since $\Phi_{\sO_{\Delta}}(\sO_{\E_1}(0; \xi^j)) \cong \sO_{\E_1}(0; \xi^j)$, and from \cite{bk_fm_nodal}, $\Phi_{\s{I}_\Delta[1]}(\sO_{\E_1}(0; \xi^j)) \cong \sO_{\E_1}(0; \xi^j)$ we must have $\Phi_{\sO_{E_1 \times E_1}}(\sO_{\E_1}(0; \xi^j)) \cong 0$ or a self extension.  However,
\begin{displaymath}
\mr{Hom}_{\E_1}(k(s), \Phi_{\sO_{E_1 \times E_1}}(\sO_{\E_1}(0; \xi^j))) \cong \mr{Hom}_{\E_1}(\sO_{\E_1}, \sO_{\E_1}(0; \xi^j))
\end{displaymath}
 for any closed $s \in \E_1$.  This implies that if $j \neq 0$, $\Phi_{\sO_{\E_1 \times \E_1}}(\sO_{\E_1}(0; \xi^j)) \cong 0$. The result then follows from the previous section.
\end{proof}
Let
\begin{displaymath}
\hat B: (\pi_n \times \pi_n)_\ast \sO_{\E_n \times \E_n} \otimes (\pi_n \times \pi_n)_\ast \sO_{\E_n \times \E_n} \rightarrow (\pi_n \times \pi_n)_\ast \sO_{\E_n \times \E_n}  
\end{displaymath}
be the monoid structure morphism.  Using the same arguments as in the previous section, if $\hat B_{k_0, k_1; j_0, j_1}$ denotes the restriction of $\hat B$ to the summand
\begin{displaymath}
(\rho^{1\ast}(\sO_{\E_1}(0; \xi^{k_0}) \otimes \rho^{2\ast}(\sO_{\E_1}(0; \xi^{j_0})) \otimes (\rho^{1\ast}(\sO_{\E_1}(0; \xi^{k_1})\otimes \rho^{2\ast}(\sO_{\E_1}(0; \xi^{j_1}))
\end{displaymath}
 Lemma \ref{notouch} implies $\hat B_{k_0, k_1; j_0, j_1}$ gives a isomorphism
 \begin{align*}
(\rho^{1\ast}(\sO_{\E_1}(0; \xi^{k_0}) \otimes \rho^{2\ast}(\sO_{\E_1}(0; \xi^{j_0})) \otimes (\rho^{1\ast}(\sO_{\E_1}(0; \xi^{k_1})\otimes \rho^{2\ast}(\sO_{\E_1}(0; \xi^{j_1})) \qquad \quad  \\ \xrightarrow{\cong} \rho^{1\ast}(\sO_{\E_1}(0; \xi^{k_0 + k_1})\otimes \rho^{2\ast}(\sO_{\E_1}(0; \xi^{j_0 +  j_1})   
 \end{align*}
In fact, it is clear we can assume $\hat{B}_{k_0, k_1; j_0, j_1} = \rho^{1\ast} B_{k_0, k_1} \otimes \rho^{2\ast} B_{j_0, j_1}$, and in general 
\begin{displaymath}
\hat B_{k_0, k_1, \ldots, k_s; j_0, j_1, \ldots j_s} = \rho^{1\ast} B_{k_0, k_1, \ldots, k_s} \otimes \rho^{2\ast} B_{j_0, j_1, \ldots, j_s} \quad s \in \mN.
\end{displaymath}
Here we are using the same notation as in \S~\ref{littlemod}.  The $\hat B_{k_0, k_1, \ldots, k_s; j_0, j_1, \ldots j_s}$ satisfy composition equalities identical to the composition equalities on $B_{i_0, \ldots, i_s}$ detailed in the previous section.  It is also clear that $\hat B_{0, 0; 0, 0}$ is the canonical identification.

With the structure morphisms set, for the same reasons as in \S \ref{littlemod}, $(\iota_n^p \times \iota_n^q)_\ast$ acts on $(\pi_n \times \pi_n)_\ast \sO_{\E_n \times \E_n}$-mod as follows: if $\sF \in \Coh(\E_n \times \E_n)$ and $M_{\sF}$ is an associated $(\pi_n \times \pi_n)_\ast \sO_{\E_n \times \E_n}$-module with structure morphism $m$, if we write $m = \Sigma_{0 \leq k, j < n} m_{k,j}$ where $m_{k, j}$ is $m$ restricted to the summand $\rho^{1 \ast} \sO_{\E_1}(0; \xi^k) \otimes \rho^{2\ast} \sO_{\E_1}(0; \xi^j) \otimes M_{\sF}$, then $(\iota_n^p \times \iota_n^q)_\ast \sF$ has as associated module $M_{\sF}$ (as an object in $\Coh(\E_1 \times \E_1)$) with structure morphism $\hat m = \Sigma_{0 \leq k, l < n} \hat m_{k,l}$ where $\hat m_{k,l} = \xi^{pk + ql}m_{k,l}$.

\section{Compatible autoequivalences of \bderive{\E_n}}
\label{gamma_chapter}
In this section we will apply Theorem \ref{stable_preserving} to the example of $n$-gons.  This serves the purpose of showing how the conditions of Theorem \ref{stable_preserving} many times reduce to easily applicable and natural conditions. 

\subsection{Reduction of conditions in Theorem \ref{stable_preserving} for $\sigma_{cl}(n)$}\label{reductions}

\begin{lemma}
\label{commute_with_iota}
Given $\Phi \in \mr{Aut}(\bderive{\E_n})$, suppose $\Phi\circ\iota_n^{\ast} \cong \iota_n^{q\ast} \circ \Phi$ for some $0 < q \leq n$, then $\Phi$ preserves the kernel of $Z_{cl}$.
\end{lemma}
\begin{proof}
By definition, an element $t \in K(\bderive{E_{n}})$ is in the kernel of $Z_{cl}$ if and only if  $t = a_{1}e_1 + \ldots a_{n}e_n$ with $\Sigma a_i = 0$.  We have $[\iota_n^{\ast}]e_i = [\iota_n^{\ast}][\eta_{\ast}\s{O}_{\mP_{i}}(-1)] =[\eta_{\ast}\s{O}_{\mP_{i+1}}(-1)] = e_{i+1}$ (if $i = n$ then $i + 1  = 1$). Thus $[\iota^{\ast}_{n}]$ acts as the identity on $e_0$, and cyclically permutes $\{e_i\}_{\mZ/n\mZ}$.  Since $\rktot = \Sigma_{0 < i \leq n} e_{i}^{\vee}$ and $\chi = e_{0}^{\vee}$ it is clear that $[\iota_{n}^{\ast}] \cdot \rktot = \rktot$ and $[\iota_{n}^{\ast}] \cdot \chi = \chi$.  Clearly the same holds for powers of $\iota^{\ast}_n$.

Using the commuting relation of $\Phi$ and $\iota_{n}^{\ast}$, we can write 
\begin{displaymath}
[\Phi]t = \Sigma_{0 < i \leq n} a_{i}[\iota_{n}^{q(i-1)\ast}]([\Phi]e_1).
\end{displaymath}  Applying $\rktot$ to the left and right yields
\begin{align*}
 \rktot([\Phi](t)) & = \rktot(\Sigma_{0 < i \leq n} a_{i}[\iota_{n}^{q(i-1)\ast}][\Phi]e_1) \\
 						& = \Sigma_{0 < i \leq n} \rktot(a_i[\iota_{n}^{q(i-1)\ast}][\Phi]e_1) \\
 						& = \Sigma_{0 < i \leq n} (a_i \rktot([\iota_{n}^{q(i-1)\ast}][\Phi]e_1)) \\
 						& = \Sigma_{0 < i \leq n} (a_i \rktot([\Phi]e_1))\\
						& = (\Sigma_{0 < i \leq n} a_i) * \rktot([\Phi]e_1) \\
						& = 0 * \rktot([\Phi]e_1) \\
 						& = 0.
\end{align*}
These equations only used the additive property and invariance of $\rktot$ under $\iota_{n}^{\ast}$.  Thus the same will be true for $\chi$ and $[\Phi]t$ is in the kernel of $Z_{cl}$.
\end{proof}

As an initial test of compatibility, we have the the following:
\begin{lemma}
\label{must be integer}
Assume $\rank(\image(Z)) = 2$ and $\image(Z)\otimes\mR = \mC$.  If $\Phi$ satisfies Theorem \ref{stable_preserving}(iii) then $[[\Phi]] \in \mr{SL}(\image(Z))$. 
\end{lemma}
\begin{proof}
Choose a basis $z_0, z_1 \in \mC$ of $\image(Z)$.  Then $z_0, z_1 \in \mathbb{H}^\prime$ and forms a basis of $\mC$ showing that $[[\Phi]]$ extends to a $\mR$-linear automorphism of $\mC$.  Since the choice of heart gives a basis for $K(\s{T})$ consisting of elements of the form $[F]$ with $F \in \sA$, we can assume $z_0, z_1 \in \image(Z)_{eff\cap comp}$ with $\phi_f(z_0) < \phi_f(z_1)$.

With our choice of basis, the half plane decided by the line $\mR \cdot z_0$ and $z_1$ does not intersect the positive real axis.  If $[[\Phi]]$ satisfies Theorem \ref{stable_preserving}(iii), the same will be true for the half plane decided by $R \cdot [[\Phi]](z_0)$ and $[[\Phi]](z_1)$. This shows the orientation on $\{[[\Phi]](z_0), [[\Phi]](z_1)\}$ imparted by the log branch agrees with the orientation as elements in $\mC$.  Thus, $[[\Phi]]\otimes \mR$ preserves orientation so $[[\Phi]] \in \mr{SL}(\image{Z})$.
\end{proof}

\subsection{The autoequivalence group $\mr{Aut}_{cl}(n)$}

Using these reductions we will produce the group $\Aut_{cl}(n) \subset \Aut(\bderive{\E_n})$ of autoequivalences compatible with $\sigma_{cl}(n)$.  We obtain this subgroup by explicitly constructing autoequivalences of $\bderive{\E_n}$.  This group will be an extension of $\Gamma_0(n) \subset \mr{SL}(2, \mZ)$, the congruence subgroup consisting of elements that are upper triangular under reduction of coefficients $\mr{SL}(2, \mZ) \rightarrow \mr{SL}(2, \mZ/n\mZ)$.
We begin by lifting autoequivalences of $\bderive{\E_1}$ to endomorphisms of $\bderive{\E_n}$.  Once this shown, we ascertain that these endomorphisms are strongly compatible autoequivalences.  As in \S \ref{n-gons_chapter}, our base field will be $\mC$ and $\xi$ will be a primitive $n$th root of unity. 

\subsubsection{The kernels $K^A$,  $A \in \mr{SL}(2, \mZ)$}
The lifting of autoequivalences relies on lifting certain objects in $\bderive{\E_1 \times \E_n}$ to objects in $\bderive{\E_n \times \E_n}$.  The underlying reason we can do this centers on a symmetry of our chosen objects.  This symmetry is the second statement of Lemma \ref{coherentthis}, and allows us to realize a $(\pi_{n} \times \mr{Id})_\ast \sO_{\E_n \times \E_n}$-module structure on the pullback $(\mr{Id} \times \pi_n)^\ast$ of these symmetrical objects.  This is enough to show a lift.  We begin by showing a nice relation regarding integral transforms of $\bderive{\E_1}$.  This relation is well known for elliptic curves.  We then define our sheaves of interest and show how this relation gives them additional symmetry. Throughout this section, $\rho_{nm}^i$, $i = 1, 2$,  will denote the standard projection of $\E_n \times \E_m$ onto the $i$th factor.   Also, when no confusion can arise, we omit the structure sheaf from the tensor notation.  We set $\sK_{\lambda_1, \lambda_2} := \rho_{11}^{2\ast} \sO_{\E_1}(0; \lambda_1) \otimes \sO_{\Gamma_{t_{\lambda_2}}}$, with $\lambda_1, \lambda_2 \in \mC^\ast$.  Using Proposition \ref{geometric_automorphism_relation} and the invariance of $\Pic^0(\E_1)$ under translation, we see that under the convolution product (up to isomorphism), the $\sK_{\lambda_1, \lambda_2}$ form an abelian group isomorphic to $\mC^\ast \times \mC^\ast$.  


Let $p$ be a smooth closed point of $\E_1$.  In \cite{bk_fm_nodal} the sheaf  $\sP := \rho_{11}^{1\ast} \sO_{\E_1}(p) \otimes \s{I}_{\Delta}[1] \otimes \rho_{11}^{2\ast} \sO_{\E_1}(p) \in \Coh(\E_1 \times \E_1)$ was shown to have similar properties to the dual of the Poincar\'e sheaf of an elliptic curve.  In particular, it is known that if $\hat \sP := \sP^\vee[1]$, then  $\sP * \hat \sP \cong \sO_{\Delta}$.  Let 
\begin{align*}
\sK_h & := \sO_{\Delta}\otimes  \rho_{11}^{2\ast} \sO_{\E_1}(-p) \\  
\hat{\sK_h} & := \sO_{\Delta} \otimes \rho_{11}^{2\ast} \sO_{\E_1}(p)\\
\sK_v & := \sP * \sK_h * \hat \sP \\
\hat{\sK_v} & := \sP * \hat \sK_h * \hat{\sP}
\end{align*}
Both $\Phi_{\sK_h}$ and $\Phi_{\sK_v}$ are autoequivalences and it is clear 
\begin{align*}
 [\sK_h] \cong \left[ \begin{array}{cc} 1  & -1 \\ 0  & 1 \end{array} \right]& \qquad [\Phi_{\sP}] \cong \left[ \begin{array}{cc} 0  & 1 \\ -1  & 0 \end{array} \right] &  [\sK_v] \cong \left[ \begin{array}{cc} 1  & 0 \\ 1  & 1 \end{array} \right] 
\end{align*}


\begin{lemma}
\label{vertical}
  $\sK_{\lambda_1,\lambda_2} * \sK_h \cong \sK_h * \sK_{\lambda_1\lambda_2, \lambda_2}$
\end{lemma}
\begin{proof}
This is equivalent to showing $\hat{\sK_h} * \sK_{\lambda_1, \lambda_2} * \sK_h \cong \sK_{\lambda_1 \lambda_2,\lambda_1}$.  This is essentially Proposition \ref{geometric_automorphism_relation}.  Using the graph morphism $\E_1 \xrightarrow{\Gamma_{t_{\lambda_1}}} \E_1 \times \E_1$ and proper base change
  \begin{align*}
    \hat{\sK_h} * \sK_{\lambda_1,\lambda_2} * \sK_h & \cong \rho_{1,1}^{1\ast} \sO_{\E_1}(p) \otimes \sK_{\lambda_1,\lambda_2} \otimes \rho_{1,1}^{2\ast} \sO_{\E_1}(-p) \\
   & \cong \rho_{1,1}^{1\ast} \sO_{\E_1}(p) \otimes \Gamma_{t_{\lambda_2}\ast}(\sO_{\E_1}) \otimes \rho_{1,1}^{2\ast} \sO_{\E_1}(-p) \otimes \rho_{1,1}^{2\ast} \sO_{\E_1}(0; \lambda_1) \\ 
   & \cong \Gamma_{t_{\lambda_2}\ast}(\sO_{\E_1}(p))  \otimes \rho_{1,1}^{2\ast} \sO_{\E_1}(-p) \otimes \rho_{1,1}^{2\ast} \sO_{\E_1}(0; \lambda_1) \\
   & \cong \Gamma_{t_{\lambda_2}\ast}(\sO_{\E_1})) \otimes \rho_{1,1}^{2\ast}(\sO_{\E_1}(t_{\lambda_2}(p)) \otimes \sO_{\E_1}(-p) \otimes \sO_{\E_1}(0; \lambda_1))
\end{align*}
 Since $\sO_{\E_1}(t_{\lambda_{2}}(p) - p) \cong \sO_{\E_1}(0; \lambda_2)$ by definition,  
\begin{align*}
\Gamma_{t_{\lambda_2}\ast}(\sO_{\E_1})) \otimes \rho_{1,1}^{2\ast}(\sO_{\E_1}(t_{\lambda_2}(p)) \otimes \sO_{\E_1}(-p) \otimes \sO_{\E_1}(0; \lambda_1)) \qquad \qquad \\ \cong \sO_{\Gamma_{t_{\lambda_2}}} \otimes \rho_{1,1}^{2\ast} (\sO_{\E_1}(0; \lambda_2) \otimes \sO_{\E_1}(0; \lambda_1)) 
 \end{align*}
and this last object is $\sK_{\lambda_1 \lambda_2, \lambda_2}$, thus the result.
\end{proof}


\begin{lemma}
$\sP * \sK_{\lambda_1, \lambda_2} * \hat{\sP} \cong \sK_{\lambda^{-1}_2, \lambda_1}$
\end{lemma}
\begin{proof}
As autoequivalences of $\bderive{\E_1}$, this is a consequence of the fact
\begin{align*}
  \Phi_{\sK_{\lambda_1, \lambda_2}} \cong T_{k(t_{\lambda_1}(p))} \circ T_{k(p)}^{-1} \circ T^{-1}_{\sO_{\E_1}(0; \lambda_2)} \circ T_{\sO_{\E_1}}
\end{align*}
combined with the isomorphism of autoequivalences,
\begin{align*}
  \Phi_{\hat \sP} \circ T_\sS \circ \Phi_{\sP} \cong T_{\Phi_{\hat \sP}(\sS)}
\end{align*}
and that $\Phi_{\hat \sP}(k(t_\lambda)(p)) \cong \sO_{\E_1}(0; \lambda)$, $\Phi_{\hat \sP}(\sO_{\E_1}(0; \lambda)) \cong k(t_{\lambda^{-1}}(p))[-1]$. Here $T_{\sF}$ is the spherical twist associated to a spherical object $\sF \in \bderive{\E_1}$.  These relations are all easily derived from \cite{bk_fm_nodal}, and the isomorphism $\Phi_{\sP} \cong T_{k(p)} \circ T_{\sO_{\E_1}} \circ T_{k(p)}$.  To prove it on the kernel level, one uses \cite[Remark 8.9]{ballard}.  This can be also be explicitly verified.
\end{proof}

\begin{corollary}
  $\sK_{\lambda_1, \lambda_2} * \sK_v \cong \sK_v * \sK_{\lambda_1, \lambda^{-1}_1 \lambda_2}$
\end{corollary}
\begin{proof}
  Conjugate Lemma \ref{vertical} by $\sP$.
\end{proof}

Using our standard basis on $K_0(\E_1)$, let 
\begin{displaymath}
  A = \left[ \begin{array}{cc} a  & b \\ c  & d \end{array} \right] \in \mr{SL}(2, \mZ).
\end{displaymath} 
The explicit description of $[\Phi_{\sK_h}]$ and $[\Phi_{\sK_v}]$ above implies 
\begin{displaymath}
 A = [\Phi_{\sK_{a_m}}]^{b_m} \circ [\Phi_{\sK_{a_{m - 1}}}]^{b_{m-1}} \circ \ldots \circ [\Phi_{\sK_{a_{1}}}]^{b_1} \circ [\Phi_{\sK_{a_0}}]^{b_0}  
\end{displaymath}
for some $m \in \mN$ and $a_i \in \{v, h \}$ and $b_i \in \{-1, 1\}$.  Letting $\sK_v^{-1}:= \hat{\sK_v}$ and $\sK_h^{-1}:= \hat{\sK_h}$, we set $\sK^A = \sK^{b_0}_{a_0} * \sK^{b_1}_{a_1} * \ldots * \sK^{b_m}_{a_m}$, we have $[\sK^A] = A$ and the following commuting relation.  
 
\begin{lemma}
 \label{commutingactions}
 $\sK_{\lambda_1,\lambda_2} * \sK^A \cong   \sK^A * \sK_{\lambda_1^a \lambda_2^{-b}, \lambda_1^{-c} \lambda_2^d}$
\end{lemma}
\begin{proof}
Let
\begin{displaymath}
 w = \left[ \begin{array}{c} \lambda_1 \\ \lambda_2 \end{array} \right].
\end{displaymath}
and $\sK_w := \sK_{\lambda_1, \lambda_2}$.  Writing the group structure of $\mC^\ast$ additively, Lemma \ref{vertical} implies $\sK_w * \sK_h \cong \sK_h * \sK_{w^\prime}$ where $w^\prime = [\Phi_{\sK_h}]^{-1}w$.  A similar statement holds for $\sK_v$.  Repeated application of this principle shows $\sK_w * \sK^A \cong \sK^A * \sK_{w^{\prime \prime}}$ where 
\begin{displaymath}
  w^{\prime \prime} = [\Phi_{\sK_{a_m}}]^{-1^{b_m}} [\Phi_{\sK_{a_{m - 1}}}]^{-1^{b_{m-1}}} \ldots  [\Phi_{\sK_{a_{1}}}]^{-1^{b_1}} [\Phi_{\sK_{a_0}}]^{-1^{b_0}}w
\end{displaymath}
If we let $B = [\Phi_{\sK_{a_m}}]^{-1^{b_m}} [\Phi_{\sK_{a_{m - 1}}}]^{-1^{b_{m-1}}} \ldots  [\Phi_{\sK_{a_{1}}}]^{-1^{b_1}} [\Phi_{\sK_{a_0}}]^{-1^{b_0}}$, then explicit calculation shows   
\begin{align*}
B = P(A^T)^{-1} P \textrm{, where }  P = \left[ \begin{array}{cc} 0  & 1 \\ 1  & 0 \end{array} \right].\; \; & \textrm{It follows that }  B = \left[ \begin{array}{cc} a  & -b \\ -c  & d \end{array} \right].
\end{align*} 
Converting back to multiplicative notation yields the result.
\end{proof}

\begin{lemma}\mbox{}
\label{coherentthis}
  \begin{enumerate} 
    \item $\sK^A[M] \in \Coh(\E_1 \times \E_1)$ for some $M \in \mZ$.
    \item If $A \in \Gamma_0(n)$, $\rho_{11}^{1\ast} \sO_{\E_1}(0; \xi^k) \otimes \sK^A \cong \sK^A \otimes \rho_{11}^{2\ast} \sO_{\E_1}(0; \xi^{ak})$.
\end{enumerate}
\end{lemma}
\begin{proof}
 {\it 1.}  By construction $\Phi_{\sK^A}$ is an autoequivalence.  As shown in \cite{bk_fm_nodal}, $\mr{Stab}(\E_1) \cong \widetilde{\mr{GL}}^{+}(2, \mR)$, the isomorphism is provided by $\sigma_{cl}(1)$ and the natural right action on the stability manifold.  The left action by $\Aut(\bderive{\E_1})$ implies all autoequivalences are strongly compatible with all stability conditions.  Since $k(p)$, $p \in \E_1$ is stable $\sV_{p}[M] := \Phi_{\sK^A}(k(p))$ is a shifted stable indecomposable locally free sheaf (torsion free sheaf) for $p$ smooth (singular), with $M$, $\rank(\sV_{p}) = n$ and $\chi(\sV_{p}) = d$ uniform for all $p$.  Replacing $\sK^A$ with $\sK^A[-M]$, we can assume that $M = 0$ and by \cite[Lemma 3.31]{MR2244106} that $\sK^A$ is a sheaf.

{\it 2.}  Setting $c = nm$, we have the following string of isomorphisms following from Lemma \ref{commutingactions}:
\begin{align*}
  \rho_{11}^{1\ast} \sO_{\E_1}(0; \xi^k) \otimes \sK^A & \cong \sK_{\xi^k, 1} * \sK^A \\
 & \cong \sK^A * \sK_{\xi^{ak}, \xi^{-ck}} \\
& \cong \sK^A * \sK_{\xi^{ak}, \xi^{-nmk}} \\
& \cong \sK^{A} * \sK_{\xi^{ak}, 1} \\
&\cong \sK^A \otimes \rho_{11}^{2\ast} \sO_{\E_1}(0; \xi^{ak})
\end{align*}
\end{proof}

\begin{proposition}
\label{monodromy_lift}
If $A \in \Gamma_0(n)$, then there exists $\sK^A_n \in \bderive{\E_n \times \E_n}$ with $(\pi_n \times \mr{Id})_\ast \sK^A_n  \cong (\mr{Id} \times \pi_n)^\ast \sK^A$.  
\end{proposition}
\begin{proof}
Similar to \S \ref{mod_structure}, we will use the notation $\widetilde \pi^{kl}_{nm}$ for the morphism $\pi_{k, n} \times \pi_{l, m}$.  By Lemma \ref{coherentthis}, we can assume $\sK^A$ is a coherent sheaf on $\E_1 \times \E_1$.  To prove the proposition, we will show $\wtpi^{1n}_{11\ast}\wtpi^{1n\ast}_{11} \sK^A$ has a natural $\wtpi^{nn}_{11\ast} \sO_{\E_n \times \E_n}$-module structure that restricts to its natural $\wtpi^{1n}_{11\ast} \sO_{\E_1 \times \E_n}$-module structure under the inclusion $\wtpi^{1n}_{11\ast} \sO_{\E_1 \times \E_n} \hookrightarrow \wtpi^{nn}_{11\ast} \sO_{\E_n\times \E_n}$.  We then set $\sK^A_n$ to be the coherent sheaf on $\E_n \times \E_n$ associated to this $\wtpi^{nn}_{11\ast} \sO_{\E_n\times \E_n}$-module.

Using methods similar to Lemma \ref{notouch},
\begin{displaymath}
  \wtpi^{1n}_{11\ast} \sO_{\E_1 \times \E_n} \cong \oplus_{0 \leq j < n} \rho_{11}^{2\ast} \sO_{\E_1}(0; \xi^{j})
\end{displaymath}
and
\begin{displaymath}
\sK^{A\prime} := \wtpi^{1n}_{11\ast} \wtpi^{1n\ast}_{11} \sK^A \cong  \sK^A \otimes \oplus_{0 \leq k < n} \rho_{11}^{2\ast} \sO_{\E_1}(0; \xi^k) 
\end{displaymath}
The natural $\wtpi^{1n}_{11\ast} \sO_{\E_1 \times \E_n}$-module structure on $\sK^{A\prime}$ is supplied by $\mr{Id} \otimes \rho_{11}^{2\ast}B_{ij} = \mr{Id} \otimes \hat B_{0, 0; i, j}$ (in the notation of \S \ref{littlemod}).  Combining this with Lemma \ref{notouch}, we see 
\begin{align*}
  \pi^{nn}_{11\ast} \sO_{\E_n \times \E_n} \otimes \sK^{A\prime}  \cong \oplus_{0 \leq i,j, k < n} \rho_{1,1}^{1\ast} \sO_{\E_1}(0; \xi^i) \otimes \rho_{1,1}^{2\ast} \sO_{\E_1}(0; \xi^j) \otimes \sK^A \otimes \rho_{1,1}^{2\ast} \sO_{\E_1}(0; \xi^k)
\end{align*}
We will define $m: \wtpi^{nn}_{11\ast} \sO_{\E_n\times \E_n} \otimes \sK^{A\prime} \to \sK^{A\prime}$ by giving its restriction $m_{ijk}$ to each summand in the above decomposition.

From Lemma \ref{coherentthis}, 
\begin{align*}
  \phi_{\xi}: \rho_{11}^{1\ast} \sO_{\E_1}(0; \xi) \otimes \sK^A
  \xrightarrow{\cong} \sK^A \otimes \rho_{11}^{2\ast} \sO_{\E_1}(0;
  \xi^{a})
\end{align*}
We set $\phi_{\xi^{k}} := \hat B_{0, \ldots, 0; a, \ldots, a} \circ \phi_{\xi} \circ \ldots \circ \phi_{\xi} \circ \hat B^{-1}_{1, \ldots, 1; 0, \ldots 0}$ (in each grouping of the subscript, there are $k$ repeated elements). Then 
\begin{align*}
  \phi_{\xi^k}: \rho_{11}^{1\ast} \sO_{\E_1}(0; \xi^k) \otimes \sK^A
  \xrightarrow{\cong} \sK^A \otimes \rho_{11}^{2\ast} \sO_{\E_1}(0;
  \xi^{ak})
\end{align*}
and the $\phi_{\xi^k}$ satisfy a compatibility equality: 
\begin{align}
\label{compat_eq}
  \phi_{\xi^{k + l}} = \hat B_{0, 0; ak, al} \circ \phi_{\xi^k} \circ \phi_{\xi^l} \circ \hat B^{-1}_{k,  l; 0, 0}
\end{align}

Define $m_{ijk} = \hat B_{0, 0, 0; ai, k,j} \circ \phi_{\xi^i}$.  It is clear
\begin{align*}
  m_{ijk}: \rho_{11}^{1\ast} \sO_{\E_1}(0; \xi^i)\otimes \rho_{11}^{2\ast} \sO_{\E_1}(0; \xi^{j}) \otimes \sK^A \otimes \rho_{11}^{2\ast}\sO_{\E_1}(0; \xi^k) \xrightarrow{\cong} \qquad \qquad \\ \sK^A \otimes \rho_{11}^{2\ast}\sO_{\E_1}(0; \xi^{ai + j + k})
\end{align*}
Using the decomposition of $\sK^{A\prime}$ we can consider $m_{ijk}$ to be a morphism to $\sK^{A\prime}$.

Now that we have defined $m$, we must show that it makes $\sK^{A\prime}$ into a $\wtpi^{nn}_{11\ast} \sO_{\E_n\times \E_n}$-module, i.e., $m \circ \hat B = m \circ m$.  To show this, it is enough to restrict to the summands of $\wtpi^{nn}_{11\ast} \sO_{\E_n\times \E_n} \otimes \wtpi^{nn}_{11\ast} \sO_{\E_n\times \E_n} \otimes \sK^{A\prime}$ of the form $\rho_{1,1}^{1\ast} \sO_{\E_1}(0; \xi^{k_1}) \otimes \rho_{1,1}^{2\ast} \sO_{\E_1}(0; \xi^{k_0}) \otimes \rho_{11}^{2\ast}\sO_{\E_1}(0; \xi^{j_0}) \otimes \rho_{11}^{2\ast}\sO_{\E_1}(0; \xi^{j_1}) \otimes \sK^A \otimes\rho_{11}^{2\ast}\sO_{\E_1}(0; \xi^{l})$.  This amounts to showing
\begin{align*}
  m_{k_1, j_1, ak_0 + j_0 + l} \circ m_{k_0,j_0, l} = m_{k_1 + k_0, j_0 + j_1, l} \circ \hat B_{k_1, k_0; j_0, j_1}.
\end{align*}
Temporarily, we abbreviate $\rho_{1,1}^{i\ast} \sO_{\E_1}(0; \xi^k)$ as $\bar \xi_i^k $ and have an implicit tensor product. All identifications $\bxi_1^{k_0}\bxi_1^{k_1}\bxi_{2}^{j_0}\bxi_2^{j_1} \to \bxi_1^{k_0 + k_1}\bxi_2^{j_0 + j_1}$ in the following diagrams are done using $\hat{B}$.  With this notation, $ m_{k_1, ak_0 + j_0 + l, j_1} \circ m_{k_0,l, j_0}$ is the operation  
\begin{align*}
   \bxi_1^{k_1}  \bxi_1^{k_0}   \bxi_2^{j_0}   \bxi_2^{j_1} \sK^A  \bxi_2^{l} & \to \bxi_1^{k_1}\sK^A \bxi_2^{ak_0} \bxi_2^{l} \bxi_2^{j_0} \bxi_2^{j_1}\\
 &  \to \bxi_1^{k_1} \sK^A \bxi_2^{ak_0 + l + j_0} \bxi_2^{j_1} \\
& \to \sK^A \bxi_2^{ak_1} \bxi_2^{ak_0 + l + j_0} \bxi_2^{j_1} \\
& \to  \sK^A \bxi_2^{ak_1 + ak_0 + l + j_0 + j_1}
\end{align*}
The associativity of the $\hat B$ operation implies this is the same as  
\begin{align*}
   \bxi_1^{k_1}  \bxi_1^{k_0}   \bxi_2^{j_0}   \bxi_2^{j_1} \sK^A  \bxi_2^{l} & \to \bxi_1^{k_1}\sK^A \bxi_2^{ak_0} \bxi_2^{l} \bxi_2^{j_0} \bxi_2^{j_1}\\
 &  \to \sK^A \bxi_2^{ak_1} \bxi_2^{ak_0}\bxi_2^{l}\bxi_2^{j_0} \bxi_2^{j_1} \\
& \to  \sK^A \bxi_2^{ak_1 + ak_0 + j_0 + j_1 + l}
\end{align*}
Note this last morphism is $\rho_{11}^{2^\ast}B_{ak_1, ak_0, j_0, j_1, l}$.  On the other hand, $m_{k_1 + k_0, l, j_0 + j_1} \circ \hat B_{k_1, k_0; j_0, j_1}$ is the operation
\begin{align*}
  \bxi_1^{k_1}  \bxi_1^{k_0}   \bxi_2^{j_0}   \bxi_2^{j_1} \sK^A  \bxi_2^{l} & \to \bxi_1^{ k_1 + k_0 }\sK^A \bxi_2^{l} \bxi_2^{j_0 + j_1}\\
 &  \to \sK^A \bxi_2^{a(k_1 + k_0)} \bxi_2^{l} \bxi_2^{j_0 + j_1} \\
& \to \sK^A \bxi_2^{ak_1 + ak_0 + l + j_0 + j_1}
\end{align*}
The compatibility equation (\ref{compat_eq}) and the associativity of $\hat B$ implies this is the same as 
\begin{align*}
 \bxi_1^{k_1}  \bxi_1^{k_0}     \bxi_2^{j_0}   \bxi_2^{j_1} \sK^A  \bxi_2^{l} & \to \bxi_1^{k_1}\sK^A \bxi_2^{ak_0} \bxi_2^{l} \bxi_2^{j_0} \bxi_2^{j_1}\\
& \to \sK^A \bxi_2^{ak_1} \bxi_2^{ak_0}  \bxi_2^{j_0} \bxi_2^{j_1}\bxi_2^{l} \\
& \to  \sK^A \bxi_2^{ak_1 + ak_0 + l + j_0 + j_1}
\end{align*}
Thus the equality  $m_{k_1, j_1, ak_0 + j_0 + l} \circ m_{k_0,j_0, l} = m_{k_1 + k_0, j_0 + j_1, l} \circ \hat B_{k_1, k_0; j_0, j_1}$  is due to the associativity of $\hat B$ and the explicit choice of isomorphism $\bxi_1^{k}\sK^A \cong \sK^A \xi_2^{ak}$, i.e., the only choice was for $k=1$.  Thus $m$ makes $\sK^{A\prime}$ into a $\wtpi^{nn}_{11\ast} \sO_{\E_n\times \E_n}$-module.

We are left with showing $m$ is compatible with the natural $\wtpi_{1,1\ast}^{1,n}\sO_{\E_1 \times \E_n}$-module structure $\sK^{A\prime}$.  This is clear since the morphism $\wtpi_{1,1\ast}^{1,n}\sO_{\E_1 \times \E_n} \hookrightarrow \wtpi^{nn}_{11\ast} \sO_{\E_n\times \E_n}$ sends the summand $\rho_{11}^{2\ast}\sO_{\E_1}(0; \xi^j) \to \rho_{11}^{1\ast}\sO_{\E_1}(0; \xi^0)\otimes \rho_{11}^{2\ast}\sO_{\E_1}(0; \xi^j)$ and $m_{0,k, j} = \mr{Id}_{K^A} \otimes B_{0, 0;k, j}$.  

\end{proof}

Now that $\sK^A_n \in \bderive{\E_n \times \E_n}$ is constructed, we list some of its important properties.
\begin{lemma}
\label{monodromy_properties}
	With the same setup as Proposition \ref{monodromy_lift}.  Then $\sK^A_n$ has the following properties:
	\begin{enumerate}
	 \item $(\iota^a_n \times \iota_n)_\ast \sK^A_n \cong \sK^A_n$. 
	\item $\pi_{n\ast}(\sK^A_n|_{p \times \E_n}) \cong \sK^A|_{\pi_n(p) \times \E_1}$ for any closed $p \in \E_n$, and therefore is a torsion free sheaf.
	\item for any two distinct $p_1, p_2 \in \E_n$, $\rhom{i}(\sK^A_n|_{p_1 \times E_n}, \sK^A_n|_{p_2 \times E_n}) = 0$ for all $i \in \mZ$. 
	\end{enumerate}
\end{lemma}
\begin{proof}
{\it 1. } 
From \S~\ref{mod_structure}, $(\iota^a_n \times \iota_n)_\ast \sK^A_n$ consists of the data $M = \sK^{A\prime}$ with structure morphism $\hat m_{jkl} = \xi^{aj}\xi^{k} m_{jkl}$.  Let $\psi: M \xrightarrow{\cong} M$ be given summand wise as $\sK^A \otimes \rho_{11}^{2\ast}\sO_{\E_1}(0; \xi^k) \xrightarrow{\xi^{k}\cdot \mr{Id}} \sK^A \otimes \rho_{11}^{2\ast}\sO_{\E_1}(0; \xi^k)$.  We claim $\psi \circ  m = \hat {m} \circ (\mr{Id} \times \psi)$, and thus, after converting to $\Coh(\E_n \times \E_n)$,  $\psi$ gives an isomorphism $\sK^A_n \cong  (\iota^a_n \times \iota_n)_\ast \sK^A_n$.  
To check the equality, it suffices to check on our standard decomposition.  Now
\begin{displaymath}
\hat m \circ (\mr{Id} \otimes \phi)|_{\rho_{1,1}^{1\ast} \sO_{\E_1}(0; \xi^j) \otimes \rho_{1,1}^{2^\ast} \sO_{\E_1}(0; \xi^k) \otimes \sK^A \otimes \rho_{1,1}^{2^\ast} \sO_{\E_1}(0; \xi^l)} = \xi^{ak} \xi^{k} \xi^{l} m_{jkl}
\end{displaymath}
while 
\begin{displaymath}
(\phi \circ m)|_{\rho_{1,1}^{1\ast} \sO_{\E_1}(0; \xi^j) \otimes \rho_{1,1}^{2^\ast} \sO_{\E_1}(0; \xi^k) \otimes \sK^A \otimes \rho_{1,1}^{2^\ast} \sO_{\E_1}(0; \xi^l)} = \xi^{aj + k +l} m_{jkl}
\end{displaymath}
Thus, both sides are multiplication by the scalar $\xi^{aj + k + l}$.

\noindent{\it 2.} Since $A \in \Gamma_0(n)$, $a = d^{-1} \mod n$ and thus multiplication by $a$ is invertible ($\mod n$).  This implies the morphism $\oplus_{0 \leq k < n} \rho_{1,1}^{1\ast}\sO_{\E_1}(0; \xi^k) \otimes \sK^A \xrightarrow{\Sigma \phi_{\xi^{k}}} \oplus_{0 \leq l < n} \sK^A \otimes \rho_{1,1}^{2\ast}\sO_{\E_1}(0; \xi^{l})$ is an isomorphism and  
\begin{displaymath}
\sK^{A\prime \prime} :=  \oplus_{0 \leq k < n}\rho_{11}^{1\ast}\sO_{\E_1}(0; \xi^k) \otimes \sK^A \cong (\pi_n \times \mr{Id})_\ast (\pi_n \times \mr{Id})^\ast \sK^A \cong \sK^{A\prime}
\end{displaymath}
it is an easy computation to see that the natural $(\pi_n \times \mr{Id})_\ast \sO_{\E_n \times \E_1}$-module structure on $\sK^{A\prime\prime}$ is the restriction of $m$ to $\Sigma_{0 \leq k < n} m_{k,0, l}$ under the above isomorphism.
Thus $(Id \times \pi_n)_\ast \sK^A_n \cong (\pi_n \times \mr{Id})^\ast \sK^A$ as well.  The result follows since $\pi_{n\ast}(\sK^A_n|_{p \times E_n}) \cong (\mr{Id}\times \pi_n)_\ast \sK^A_n|_{p \times \E_1} \cong (\pi_n \times \mr{Id})^\ast \sK^A|_{p \times \E_1} \cong \sK^A|_{\pi_n(p) \times \E_1}$.

\noindent {\it 3.}  This follows from the inclusion
\begin{align*}
	\rhom{j}_{\E_n}(\sK^A_n|_{p_1 \times E_n}, \sK^A_n|_{p_2 \times E_n}) &\hookrightarrow  \rhom{j}_{\E_n}(\sK^A_n|_{p_1 \times E_n}, \oplus_{ 0 \leq l < n} \iota_{n}^{l\ast}(\sK^A_n|_{p_2 \times E_n})) \\
	& \cong \rhom{j}_{\E_1}(\pi_{n\ast}(\sK^A_n|_{p_1 \times E_n}), \pi_{n\ast}(\sK^A_n|_{p_2 \times E_n})).
\end{align*}
We have two cases: $p_{2} \neq \iota_{n}^{k} p_1$ and $p_{2} = \iota_{n}^{k} p_1$ for some $k$.  For the former, it follows from the original family $\s{K}^A$;  for the latter, it follows from the stability properties of the fibers of $\sK^A$
\end{proof}

\begin{lemma}
\label{flatty}
$\sK^n_A$ is a sheaf and flat over $\E_n$ with regards to projection onto the first factor.
\end{lemma}
\begin{proof}
From Lemma \ref{monodromy_properties}, $\Phi_{\sK^A_n}(k(p))$ is a torsion free sheaf for all $p \in \E_n$.  However, letting $\rho_i$ denote the projections of $\E_n \times \E_n$, we have
\begin{align*}
\Phi_{\sK^A_n}(k(p))  &= R\rho_{2\ast}(\rho_{1}^{\ast}k(p)\otimes^{L}\sK^A_n) \\
					&\cong  R\rho_{2\ast}(i_{\ast}\sO_{p\times \E_n}\otimes^L \sK^A_n) \\
					&\cong  R\rho_{2\ast}i_{\ast}(Ri^{\ast}\sK^A_n)
\end{align*}
But, $i \circ \eta_2 = \mr{Id}$. Thus $\Phi_{\sK^A_n}(k(p)) \cong Ri^{\ast}\sK^A_n$.  Since the left side is a torsion free sheaf uniformly concentrated in a single cohomological degree for all $p$, the result follows from \cite[Lemma 3.31]{MR2244106}.
\end{proof}

\subsubsection{The autoequivalence $\Phi_{\sK_n^A}$}

Throughout this section, we assume $A \in \Gamma_0(n)$.  It is now clear that $\sK^A_n$ inherits many nice properties from $\sK^A$.  It is not surprising then that there is a nice relationship between $\Phi_{\sK^A}$ and $\Phi_{\sK^A_n}$.
\begin{lemma}
	\label{raised kernel commutes}
	$\Phi_{\sK^A_n} \circ \pi_n^{\ast} \cong \pi_n^\ast \circ \Phi_{\sK^A}$
\end{lemma}
\begin{proof}
	We have the following diagram
	\begin{displaymath}
		\xymatrix{
		& & \E_n \times \E_n \ar[ld]^{\phi_1} \ar[rd]^{\phi_2} & & \\
		& \E_n \times \E_1 \ar[rd]^{\gamma_2} \ar[ld]^{\gamma_1} & & \E_1 \times \E_n \ar[ld]^{\sigma_1} \ar[rd]^{\sigma_2} & \\
		\E_n \ar[d]^{\pi_n} & &\E_1 \times \E_1 \ar[lld]^{\eta_1} \ar[rrd]^{\eta_2} & & \E_n \ar[d]^{\pi_n} \\
		\E_1 & & & & \E_1
		}
	\end{displaymath}
	where every square is Cartesian.  The result is a calculation using the above diagram, base change, and the projection formula.
\begin{align*}
\pi_n^{\ast}\circ \Phi_{\sK^A}(F) &= \pi_n^{\ast}\eta_{2\ast}(\eta_1^{\ast}F\otimes \sK^A)\\
								&\cong \sigma_{2\ast} \sigma_{1}^{\ast}(\eta_1^{\ast}F\otimes \sK^A) \\
								&\cong \sigma_{2\ast} (\sigma_{1}^{\ast} \eta_1^{\ast}F \otimes \sigma_1^{\ast}\sK^A) \\
								&\cong \sigma_{2\ast} (\sigma_{1}^{\ast} \eta_1^{\ast}F \otimes \phi_{2\ast}\sK^A_n) \\
								&\cong \sigma_{2\ast} \phi_{2\ast}(\phi_2^{\ast}\sigma_{1}^{\ast} \eta_1^{\ast}F \otimes \sK^A_n) \\
								&\cong \sigma_{2\ast} \phi_{2\ast}(\phi_1^{\ast}\gamma_2^{\ast} \eta_1^{\ast}F \otimes \sK^A_n) \\
								&\cong \sigma_{2\ast} \phi_{2\ast}(\phi_1^{\ast}\gamma_1^{\ast} \pi_n^{\ast}F \otimes \sK^A_n) \\
								&\cong \Phi_{\sK^A_n}(\pi_n^{\ast}F).
\end{align*}
\end{proof}

With this initial analysis done, we now show $\Phi_{\sK_n^A}$ is strongly compatible with $\sigma_{cl}(n)$.  In order to apply Theorem \ref{stable_preserving}, we must first know $\Phi_{\sK_n^A}$ is an equivalence.  Once this is shown, it is simply a matter of verifying the reduced conditions derived in the previous section.

\begin{proposition}
$\Phi_{\sK_n^A}$ is an equivalence.
\end{proposition}
\begin{proof}

Let $\sK^A  \in \bderive{\E_1 \times \E_1}$ satisfy Proposition \ref{monodromy_lift}.  We have an associated $\hat \sK^A := \hat \sK_{a_m} * \hat \sK_{a_{m-1}}* \ldots * \hat \sK_{a_0}$. 
Reversing the roles of the two projections for $\hat \sK^A$, Proposition \ref{monodromy_lift} produces a (shifted) coherent sheaf $\hat \sK^{A}_n$. As noted in the proof of Lemma \ref{monodromy_properties}, this means that $\hat \sK^A_n$ satisfies $(\pi_n \times \mr{Id})_\ast \hat \sK_n^A \cong (\mr{Id} \times \pi_n)^\ast \hat \sK^A$ and Lemma \ref{raised kernel commutes} applies. 
 Lemma \ref{raised kernel commutes} gives an isomorphism of functors
\begin{displaymath}
\pi_n^\ast \circ \Phi_{\hat \sK} \circ \Phi_{\sK} \cong \Phi_{\hat \sK_n^A}\circ \pi_n^\ast \circ \Phi_{\sK} \cong \Phi_{\hat \sK_n^A} \circ \Phi_{\sK_n^A} \circ \pi_n^\ast.
\end{displaymath}
Applying these isomorphisms to $\oplus_{0 \leq i < n}k(\iota^i_n(p))$ for any $p \in \E_n$, yields the isomorphism $\Phi_{\hat \sK_n^A} \circ \Phi_{\sK_n^A}(\oplus_{0 \leq i < n} k(\iota^i_n p)) \cong \oplus_{0 \leq i < n} k(\iota^i_n p)$.  Correcting by some power of $\iota_n$ we can ensure that $\iota_n^{j\ast} \circ \Phi_{\hat \sK^A_n} \circ \Phi_{\sK_n^A}(k(p)) = k(p)$.  The geometric origin of our autoequivalence, the connectedness of $\E_n$,  and the fact that $\{ \iota^{k}_n (p) \}$ is a discrete subscheme combine to imply this identity holds for all $p$ (here we are using the continuity properties that our kernel affords).

Thus we know $\iota_n^{j\ast} \circ \Phi_{\hat \sK_n^A} \circ \Phi_{\sK_n^A}$ acts as the identity on $k(p)$ for all $p \in \E_n$.  By \cite[Lemma 2.11]{bk_fm_fib} we know $\iota_n^{j\ast} \circ \Phi_{\hat \sK_n^A} \circ \Phi_{\sK_n^A}$ is isomorphic to an autoequivalence obtained by tensoring by an invertible sheaf,  thus the result.
\end{proof}

\begin{proposition}
\label{they_are_strongly_compatible}
$\Phi_{\sK_n^A}$ is strongly compatible with $\sigma_{cl}(n)$ on $\E_n$.
\end{proposition}
\begin{proof}
According to Theorem \ref{stable_preserving} and the reductions carried out in \S\ref{reductions}, we need to verify that
\begin{enumerate}
\item $\coh{i}(\Phi_{\sK_n^A}(F)) = 0$ for $i \neq M, M+1$, $M \in \mZ$ and $F \in \mr{Coh}(\E_n)$
\item $\Phi_{\sK_n^A}\circ \iota_{n}^{\ast}\cong \iota_{n}^{q\ast} \circ \Phi_{\sK_n^A}$, for some $0 \leq q < n$.
\item $[\Phi_{\sK_n^A}]$ satisfies Theorem \ref{stable_preserving}(iii).
\end{enumerate}

(1) Without loss of generality, we can assume that $M = 0$.  Let $\rho_i$ designate the $i^{th}$ projection $\E_n \times \E_n \xrightarrow{\rho_i} \E_n$. By Lemma \ref{flatty}, $\sK_n^A$ is flat over $\rho_1$.  Therefore $\rho_1^{\ast}F \otimes^L \sK_n^A \cong \rho_1^{\ast}F \otimes \sK_n^A$ and the assertion reduces to showing $R\rho^{j}_{2\ast}(G) = 0$ for $G \in \mr{Coh}(\E_n \times \E_n)$, $j > 1$.  This is true since $\rho_2$ is a fibration with fiber dimension 1.

(2) Proposition \ref{geometric_automorphism_relation} shows the left side is isomorphic to $\Phi_{(\iota_n \times \mr{Id})_{\ast}\s{K}^A_n}$ and the right is isomorphic to $\Phi_{(\mr{Id} \times \iota_n^{q})^{\ast} \s{K}^A_n} \cong \Phi_{(\mr{Id} \times \iota_n)^{q\ast} \s{K}^A_n} $.  By construction $(\iota^a_n \times \mr{Id})_{\ast}\sK_n^A \cong (\mr{Id} \times \iota_n)^{-1}_\ast \sK_n^A$.  Using $(\mr{Id} \times \iota_n^{-1})_\ast \cong (\mr{Id} \times \iota_n)^\ast$ and taking the previous isomorphism to the $d$th power (where $da = 1 \mod n)$ gives the result.

(3)  From (2) we know $[\Phi_{\sK_n^A}]$ descends and $[[\Phi_{\sK_n^A}]] \in \mr{Aut}(\coimage(Z_{cl}))$.  Temporarily, we denote $Z_{cl}$ on $K(\E_n)$ as $Z_{cl}^n$.  The factorization $Z_{cl}^n = Z_{cl}^1  \circ [\pi_{n\ast}] $ and surjectivity of $[\pi_{n\ast}]$ implies $\coimage(Z^n_{cl}) \cong \coimage(Z^1_{cl})$.  We claim $[[\Phi_{\sK_n^A}]] = [[\Phi_{\sK^A}]]$ under this equivalence.  

Assuming this, explicit calculation using the standard basis introduced in \S\ref{n-gons_chapter} shows $\coimage(Z_{cl}^n)_{eff} = \coimage(Z_{cl}^1)_{eff}$.  By definition, $\coimage(Z_{cl}^n)_{eff \cap comp}$ only depends on $[[\Phi_{\sK_n^A}]]$ and thus the isomorphism $[[\Phi_{\sK_n^A}]] \cong [[\Phi_{\sK^A}]]$ implies
\begin{displaymath}
\coimage(Z_{cl}^n)_{eff\cap comp} = \coimage(Z_{cl}^n)_{eff \cap comp}
\end{displaymath}
Since $[[\Phi_{\sK^A}]]$ satisfies Theorem \ref{stable_preserving}(iii), the same will be true of $[[\Phi_{\sK_n^A}]]$.

To prove this isomorphism, we show $[[\Phi_{\sK_n^A}]]$ and $[[\Phi_{\sK^A}]]$ represent the same matrix under the choice of basis $\{[k(p)], [\sO_{\mP}(-1)]\}$ for $K(\E_1)$. Since $\coimage Z_{cl}^1 \cong K(\E_1)$, this choice will serve as a basis for $\coimage Z_{cl}^n$ and $\coimage Z_{cl}^1$ as well.  Explicit calculation shows  
 \begin{displaymath}
 [[\Phi_{\sK^A}]] \cong \left[ \begin{array}{cc} d = \chi(\sK|{p\times \E_1}) & a \\ r = \rktot(\sK|_{p \times \E_1}) & b \end{array} \right]
 \end{displaymath}
with $db - ra = 1$. It suffices to calculate $[[\Phi_{\sK_n^A}(k(q))]]$ and $\frac{[[\Phi_{\sK^A_{n}}(\sO_{\E_n})]]}{n}$ for any $p^\prime \in \E_n$ with $\pi_n(p^\prime) = p$.  This is due to the equalities $[[k(p^\prime)]] = [[k(p)]]$ and $[[\sO_{\E_n}]] = \frac{[[\sO_{\mP}(-1)]]}{n}$ under the isomorphism between coimages. 

The statement for $[[\Phi_{\sK_n^A}(k(q))]]$ is implied by $\pi_{n\ast} (\Phi_{\sK_n^A}(k(q))) \cong \Phi_{\sK}(k(p))$, which follows from the construction of $\sK_n^A$.  For the other element, by Lemma \ref{raised kernel commutes} $\Phi_{\sK_n^A}(\sO_{\E_n}) \cong \pi_n^{\ast}\circ \Phi_{\sK^A}(\sO_{\E_1})$.  Thus, $[[\Phi_{\sK_n^A}(\sO_{\E_n})]] \cong [[\pi_{n}^\ast \circ \Phi_{\sK^A}(\sO_{\E_1})]]$.  Simple calculations show
\begin{displaymath}
 [\pi_{n\ast} \circ \pi_n^\ast] \cong \left[ \begin{array}{cc} n  & 0 \\ 0 & n \end{array} \right]
\end{displaymath}
Thus, 
\begin{displaymath}
  [[\Phi_{\sK_n^A}(\sO_{\E_n})]] = \left[ \begin{array}{cc} n & 0 \\ 0 & n\end{array} \right] \left[ \begin{array}{cc} d & a \\ r & b\end{array} \right] \left[ \begin{array}{c} 0  \\ 1\end{array} \right] = \left[ \begin{array}{c} n a   \\ n b\end{array} \right]
\end{displaymath}
This yields $\frac{[[\Phi_{\sK^A_{n}}(\sO_{\E_n})]]}{n} = [[\Phi_{\sK^A}(\sO_{\E_1})]]$.  Combining these statements gives 
\begin{displaymath}
[[\Phi_{\sK_n^A}]] = \left[ \begin{array}{cc} d & a \\ r & b\end{array} \right]
\end{displaymath}
and the equivalence is shown. 

\end{proof}

\subsubsection{The group of compatible autoequivalences}
We now have a large class of compatible autoequivalences.  We are interested in the subgroup of $\Aut(\bderive{\E_n})$ containing all autoequivalences compatible with $\sigma_{cl}(n)$.  We will first define $\Aut_{cl}(n)$ and give some of its important properties.  In the next section we show that it is maximal.  

\begin{definition} Let $\Aut^{triv}(n)$ be the subgroup of $\Aut(\bderive{\E_n})$ generated by
\begin{enumerate}
\item $\otimes \pi_n^\ast\sL, \sL \in Pic^0(\E_1)$,
\item $[2] := [1] \circ [1]$, where $[1]$ is the shift functor,
\item $\Aut(\E_n)$.
\end{enumerate}
\end{definition}

\begin{definition}
\label{max_compat_subgroup}
Let $\Aut_{cl}(n)$ be the subgroup of $\Aut(\bderive{\E_n})$ generated by
\begin{displaymath}
\{\Phi \; | \Phi \cong \Phi_{\sK^A_n} \textrm{,  } A \in \Gamma_{0}(n)\}
\end{displaymath}
 and $\Aut^{triv}(n)$.
\end{definition}

\begin{lemma}
$\Aut_{cl}(n)$ is a proper subgroup of $\Aut(\bderive{\E_n})$.
\end{lemma}
\begin{proof}
For any closed smooth point $p \in \E_n$, the autoequivalence $\otimes \s{I}(p)$ will not preserve semistable objects. 
\end{proof}

 \begin{theorem}
 \label{Gamma_0}
 	All autoequivalences of $\Aut_{cl}(n)$ are strongly compatible with $(Z_{cl}, \s{P}_{cl})$.  Further, $\Aut_{cl}(n)$ is an extension
 	 \begin{displaymath}
		\xymatrix{ 1 \ar[r] & ((\mC^\ast)^n \rtimes D_{2n}) \times \mZ \times \mC^\ast \ar[r] &  \Aut_{cl}(n) \ar[r] & \Gamma_0(n) \ar[r] & 1.}
	\end{displaymath}
Under the action of $\Aut_{cl}(n)$, there are $\Sigma_{d|n, d > 0} \phi(\gcd(d, \frac{n}{d}))$ equivalence classes of phases, where $\phi$ is Euler's function.  
 \end{theorem}
\begin{proof}

It is easily calculated that all elements of $\Aut^{triv}(n)$ act trivially on $K(\E_n)$ and $\Aut^{triv}(n) \cong \Aut(\E_n) \times \mZ \times \mr{Pic}^0(\E_n) \cong ((\mC^\ast)^n \rtimes D_{2n}) \times \mZ \times \mC^\ast$.  We claim that the above generators of $\Aut_{cl}(n)$ are strongly compatible.  From Definition \ref{max_compat_subgroup} and Proposition \ref{they_are_strongly_compatible}, we are reduced to verifying this for elements in $\Aut^{triv}(n)$.  However, this is an easy application of Theorem \ref{stable_preserving}: any autoequivalence acting trivially on $K(\bderive{E_n})$ will clearly satisfy conditions (ii) and (iii).  Our specific choice of autoequivalences shows they satisfy condition (i).

Since every element in $\Aut_{cl}(n)$ satisfies Theorem \ref{stable_preserving}(ii), $Aut_{cl}(n)$ acts on $\image(Z)\cong \mZ^{2}$.  Let $m: \Aut_{cl}(n) \rightarrow \mr{GL}(2, \mZ)$ be the resulting homomorphism.  For $\E_1$, as a result of \cite{bk_fm_nodal}, any autoequivalence is of the form $\Phi_{\sK}$ for some $\sK \in \bderive{\E_1 \times \E_1}$.  By Proposition \ref{monodromy_lift} and Proposition \ref{they_are_strongly_compatible}, if $\Phi \in \Aut(\bderive{\E_1})$ and $[\Phi] \in \Gamma_0(n)$, then there exists an autoequivalence $\Phi_n$ strongly compatible with $\sigma_{cl}(n)$ lifting $\Phi$.  This shows $m(\Aut_{cl}(n)) \cong \Gamma_0(n)$.   

Let $H$ be the kernel of $m$.  To obtain the exact sequence above, we show $\Aut^{triv}(n) \cong H$.   It is clear $\Aut^{triv}(n) \subset H$.  To show it is surjective,  we write any autoequivalence $\Phi \in H$ as a composition of generators of $\Aut^{triv}(n)$.   First, if $a := \phi(\Phi(k(p))$, then $a$ is an odd integer.  For if it was not, $[[\Phi]]$ would act non-trivially on $\mZ^2$.  Let $\Psi_0 := [-a + 1] \circ \Phi$.  Then $\Phi(\sP_{cl}(1)) \subset \sP_{cl}(1)$.   As noted in Section \ref{compatible_chapter}, we then have $\Psi_0(\mr{Coh}(\E_n)) \cong  \mr{Coh}(\E_n)$ as subcategories of $\bderive{\E_n}$. Gabriel's theorem \cite{GabrielAbeliancategories} guarantees that $\Psi_0$ is generated by automorphisms of $X$ and tensoring $\mr{Coh}(\E_n)$ by invertible sheaves.  Since the latter elements are a normal subgroup of $\mr{Aut}(\mr{Coh}(\E_n))$, we can write this as a composition $(\otimes\sL) \circ \alpha_{\ast}$.  It is clear if $\sL  \notin \Aut^{triv}_{cl}(n)$, then $\otimes \sL$ will not preserve $\ker Z_{cl}$.  We have now written $\Phi$ as the composition of elements in $\Aut^{triv}(n)$.

The last statement of the theorem follows from well known computations for the number of cusps under the action of $\Gamma_0(n)$ on the upper half plane \cite{MR1291394}.  More precisely, we use the shift [2] to reduce the computation to that of equivalence classes of slopes under the action of $\Gamma_0(n)$.  In fact, a more precise statement can be made: as (redundant) representatives for the equivalence classes of phases under the action of $\Gamma_0(n)$, we can choose phases $\phi_{\frac{c}{d}}$ such that $\frac{c}{d} = -cotan(\pi\phi_{\frac{c}{d}})$ where $d|n$, $c < \frac{n}{d}$, $c,d  \in \mZ$, and $c$ coprime to $d$.  Further, the action of $\Gamma_0(n)$ on this set doesn't change the denominator of the representative.
			
\end{proof}


\section{The moduli of stable sheaves on $\E_n$}
\label{applications_chapter}
A sheaf $\sF$  is semistable (stable) in $\sigma_{cl}(n)$  if and only if $\sF$ is Simpson semistable (stable) with polarization $\s{L}(1, \ldots, 1; 1)$.  The results of \cite{Simpson_moduli}, shows the existence of a coarse moduli space  $\s{M}_{cl}(n,a)$ of $\mr{Obj}(\sP_{cl}(a) )$.   The closed points correspond to equivalence classes of objects with equivalent Jordan factors (S-equivalence).  This moduli is a projective scheme over $\mC$.  As noted in the introduction, we restrict to the case $n > 1$.

\begin{theorem}
\label{compactification}
Given a nontrivial slice $\s{P}(a)$ of $\sigma_{cl}(n)$ with $n > 1$, let $\s{M}^{st}_{cl}(n,a) \subset \s{M}_{cl}(n,a)$ denote the subscheme whose closed points correspond to stable objects and $\overline{\s{M}^{st}_{cl}(a)}$ its closure.  Then $\overline{\s{M}^{st}_{cl}(a)} \cong \E_s \coprod \mZ/n\mZ$ where $s|n$ and each component of $\overline{\s{M}^{st}_{cl}(a)}$ is a component of $\s{M}_{cl}(n,a)$.
\end{theorem}
\begin{proof}
Through the action of $\Aut_{cl}(n)$, we can assume that $a$ is one of the specific (redundant) representatives given in the proof of Theorem \ref{Gamma_0}: $a = \phi_{\frac{r}{s}}$ where $s | n$, $r < \frac{n}{s}$ and $r, s$ are coprime.  All other representatives of this form have the same denominator in the index.  Since we are primarily concerned with the denominator our choice of representative will not affect the calculation.    The coprimality of $r,s$ ensures the existence of $A \in \mr{SL}(2, \mZ)$ such that in our preferred basis
\begin{displaymath}
A = \left[\begin{array}{cc} r & * \\ s & *\end{array} \right].
\end{displaymath}
\noindent Clearly $A \in \Gamma_0(s)$. 

Proposition \ref{monodromy_lift} and Proposition \ref{they_are_strongly_compatible} construct $\Phi_{\sK_s} \in \Aut(\bderive{\E_s})$ such that for smooth $p$, $\Phi_{\sK_s}(k(p)) \in \mr{Pic}(\E_s)$.  We are interested in the functor $\pi_{n,s}^\ast \circ \Phi_{\sK_s}: \bderive{\E_s} \rightarrow \bderive{\E_n}$.  By a calculation similar to Proposition \ref{geometric_automorphism_relation}, this is equivalent to $\Phi_{\s{U}}$ where $\s{U} = (\mr{Id} \times \pi_{n,s})^{\ast} \sK_s \in \bderive{\E_s \times \E_n}$.  It is clear that for smooth $p \in \E_s$,  $\Phi_{\s{U}}(k(p)) \in \mr{Pic}(\E_n)$.  A calculation similar to the one done in Proposition \ref{they_are_strongly_compatible} shows  for all $p \in \E_s$, $\phi_f(\Phi_{\s{U}}(k(p))) = a$. We claim that $\Phi_{\s{U}}(k(p))$ is stable: if it was not stable its rank only allows for Jordan factors consisting of torsion free indecomposable subsheaves supported on strict subschemes of $\E_n$.  This gives the existence of a torsion free, but not locally free, sheaf in the Jordan decomposition of $\pi_{n,s\ast}\circ\Phi_{\s{U}}(k(p))$.  
 This is not possible since $\pi_{n,s\ast}\circ\Phi_{\s{U}}(k(p))$ is a direct sum of rank 1 stable locally free sheaves on $\E_s$.

The result of this is an inclusion $\E_{s, smooth} \stackrel{\Phi_{\s{U}}}{\rightarrow} \s{M}^{st}_{cl}(n,a)$.  We want to show that its image is the locus of all stable locally free sheaves.  We denote this latter space as $\s{M}^{st, vb}_{cl}(n, a)$.
We do this by showing that if $\sV$ is a stable locally free sheaf on $\E_n$ with $\phi(\sV) = a$ then $\iota_{n,s}^\ast(\sV) \cong \sV$ (clearly $\Phi_{\s{U}}(k(p))$ satisfies this).  Assume this to be true.  One calculates directly that $\phi(\pi_{n,s\ast}(\sV)) = a$ (in $\sigma_{cl}(s)$) and that $\pi_{n,s\ast}(\sV)$ carries a fiberwise action of $\mZ/\frac{n}{d}\mZ$.  The action gives a splitting by eigensubsheaves: $\pi_{n,s\ast}(\sV) \cong \oplus \s{W}_{\lambda_i}$, where $\lambda_i$ are the $n$th roots of unity and $\phi(\s{W}_{\lambda_i}) = a$.   One recovers $\sV$ by pulling back $\s{W}_1$.   To summarize, if  $\iota_{n,s}^\ast(\sV) \cong \sV$ then there exists a locally free sheaf $\sV^\prime$ on $\E_s$ such that $\sV \cong \pi_{n,s}^\ast\sV^{\prime}$.  The stability of $\sV$ shows $\End(\sV) \cong \mC$.  Since $\phom_{\E_s}(\sV^{\prime}, \sV^{\prime}) \subset \phom_{\E_n}(\sV, \sV)$,  $\End_{\E_n}(\sV^{\prime}) = \mC$ as well.  The isomorphism $\sP(\phi_{a}) \cong \sP(1)$ (on $\E_s$) shows $\sV^\prime$ is a stable rank one locally free sheaf.  Thus, $\s{M}^{st,vb}_{cl}(n,a) \cong \E_{s,smooth}$.   

We now show that  if $\sV$ is stable and locally free on $\E_n$ with $\phi(\sV) = a$ then $\iota_{n,s}^\ast(\sV) \cong \sV$.    Assume $\iota_{n,s}^{k\ast}\s{V} \ncong \sV$ for all  $0 < k < \frac{n}{s}$.  Since $\sV$ is assumed to be stable, $\phom_{\E_n}(\iota_{n,s}^{k\ast}\s{V}, \sV) = 0$.  Therefore,
\begin{align*}
\phom_{\E_s}(\pi_{n,s\ast}\s{V}, \pi_{n,s\ast}\s{V})& \cong \phom_{\E_n}(\pi_{n,s}^{\ast}\circ \pi_{n,s\ast}\s{V}, \s{V})\\ & \cong \phom_{\E_n}(\oplus_{0 \leq k < \frac{n}{s}} \iota_{n, s}^{k\ast}\s{V}, \s{V})\\ & \cong \mC
\end{align*}
The previous paragraph shows $\pi_{n,s\ast}\s{V}$ is stable on $\E_s$, with phase $a$.  As such, it must be locally free. This clearly cannot happen, showing $\iota_{n,s}^{k\ast}\sV \cong \sV$ for some $K \subseteq \mZ/\frac{n}{s}\mZ$.   If this is not strict, we are done.  Otherwise, using the previous paragraph again, there exists $t$ such that $s | t | n$ and $\sV \cong \pi_{n,t}^\ast \sV^\prime$ with $\sV^\prime$ locally free on $\E_t$.  Further, $\End(\sV^\prime) \cong \mC$,  $\iota_{t,s}^{k\ast}\sV^\prime \ncong \sV^\prime$ and $\phi(\sV^\prime) = a$.  Recursion gives the contradiction, showing  $\iota_{n,s}^\ast(\sV) \cong \sV$. 

The scheme $\s{M}^{st}_{cl}(n,a) \backslash \s{M}^{st,vb}_{cl}(n,a)$, by definition, consists of stable torsion free, but not locally free sheaves.   Let $\sF$ be such an object and suppose that $\sG := \pi_{n,s\ast}\sF$ is not stable (it is automatically semistable).  In the Jordan decomposition $\sG$ has a stable subsheaf $\sG^\prime$.  If $\sG^\prime$ is not locally free, there exists a $\sF^\prime$ on $\E_n$ with $\phi(\sF^\prime) = a$,  $\pi_{n,s\ast}\sF^\prime \cong \sG^\prime$, and $\sF^\prime \subset \sF$.  This contradicts that $\sF$ is stable.  Alternatively, if $\sG^\prime$ is locally free, it must be a rank one.  Thus $\sF^{\prime} := \pi_{n,s}^\ast \sG^\prime$ is a stable rank one locally free sheaf on $\E_n$.  By adjunction,  there exists a non-trivial $\sF^{\prime} \rightarrow  \sF$ between stable objects which is not an isomorphism, contradicting the stability of $\sF$ and $\sF^\prime$.

The result of this is that if $\sF$ is stable and not locally free, then $\pi_{n,s\ast}\sF$ is stable.   On $\E_s$ there are exactly $s$ of these objects (corresponding to the image of singular skyscrapers under $\Phi_{\sK_s}$).  This yields $s * \frac{n}{s}$ stable non-locally free sheaves in $\sP(a)$ on $\E_n$.   The continuous map $\s{M}_{cl}(n,a) \xrightarrow{\pi_{n,s\ast}} \s{M}_{cl}(s,a) \cong \s{M}_{cl}(s,1)$ shows that each stable  non-locally free object corresponds to an open and closed point in $\s{M}_{cl}^{st}(n,a)$.    Thus $\s{M}_{cl}^{st}(n,a) \cong \s{M}_{cl}^{st,vb} \coprod \mZ/n\mZ$.

Although we have only discussed $\Phi_{\s{U}}$ on smooth skyscrapers, it is defined on all of $\s{M}_{cl}(s,1)$, giving a map $\coprod \mr{Sym}\E_s \rightarrow \s{M}_{cl}(n,a)$.  Since $\s{M}_{cl}(n,a)$ is a separated scheme $\overline{\s{M}^{vb}_{cl}(a)} \cong \E_s$.  These moduli spaces are originating from a GIT quotients.  Since stability is an open condition, one obtains that the components of  $\overline{\s{M}_{cl}^{st}(n,a)}$  are components of $\s{M}_{cl}(n,a)$.
\end{proof}

\begin{corollary}
The group $\Aut_{cl}(n)$ contains all autoequivalences compatible with $\sigma_{cl}(n)$.
\end{corollary}
\begin{proof}
The goal will be to show that if $\Phi$ is an autoequivalence compatible with $\sigma_{cl}(n)$, then we can write $\Phi$ as the composition of elements contained within $\Aut_{cl}(n)$.

In order to carry out such a task, we must understand which phases can be the image of $\Psi(\sP(1))$.  The proof Theorem \ref{Gamma_0} gives a set of (redundant) explicit representatives for the equivalence classes of phases under the action of $\Aut_{cl}(n)$ on phase space.  For convenience we will speak of these representatives through their negative slope, not their phase (e.g. $\psi_{\frac{c}{d}} = \frac{c}{d}$).

By Theorem \ref{compactification}, for a given $a$, the scheme $\overline{\s{M}^{vb}_{cl}(n,a)} \cong \E_d$ where a representative for $a$ in the above set is $\phi_\frac{c}{d}$.  If $\Phi(\sP(1)) \subset \sP(a)$, the representative for the equivalence class $a$ must be of the form $\frac{c}{n}$ with $c$ coprime to $n$.  It is not difficult to see that $\Gamma_0(n)$ act transitively on the set of elements of this form.  Thus, we can choose the representative to be $\frac{1}{n}$, and their exists an element $\Psi_0 \in \Aut_{cl}(n)$ such that $\Psi_0 \circ \Phi(\sP(1)) \cong \sP(1)$.  The compatibility of this morphism implies that $\Psi_0 \circ \Phi(\mr{Coh}(\E_n)) \cong \mr{Coh}(\E_n)$ as subcategories of $\bderive{\E_n}$.  Thus, $\Psi_0 \circ \Phi$ is an autoequivalence of $\mr{Coh}(\E_n)$ that is extended to $\bderive{\E_n}$.  From here, we use similar methods as in the proof of Theorem \ref{Gamma_0} to give the desired result.

\end{proof}

We conclude this section by noting that much of the above analysis can be carried out in any Simpson stability condition on $\E_n$.  With the above methods it is possible to classify not only the stable objects, but also the full structure of the coarse moduli space $\s{M}_{cl}(n,a)$.   This analysis will be done in future publications.

\section*{Acknowledgments}
The author would like to thank David Ben-Zvi for generous support during this project, and the referees for helpful suggestions regarding both the mathematics and exposition of this article.

\bibliography{./advances_bib.bib}
\bibliographystyle{halpha}

\end{document}